\documentclass[10pt]{article}  
 \usepackage{amsmath}  
 \usepackage{amssymb}   
 \usepackage{amsthm}

\usepackage{hyperref}
\usepackage[utf8]{inputenc}
\usepackage[T2A]{fontenc}
\hypersetup{colorlinks,
citecolor=black,
filecolor=black,
linkcolor=black,
urlcolor=black}
\usepackage{graphics,epstopdf}
\usepackage[pdftex]{graphicx}
\usepackage{newlfont}\newlength{\defbaselineskip}
\usepackage[left=1in,right=1in,top=1in,bottom=0.8in,footskip=0.25in]{geometry}
\usepackage{setspace}
\usepackage{subfigure}

\usepackage{ulem}

\usepackage{cancel}

\allowdisplaybreaks
\newtheorem{theorem}{Theorem}[section]
\newtheorem{example}{Example}[section]
\newtheorem{lemma}{Lemma}[section]
\newtheorem{definition}[theorem]{Definition}
\newtheorem{remark}{Remark}[section]

\numberwithin{equation}{section}

\newtheorem{corollary}{Corollary}[section]

\usepackage{fancyhdr}
\usepackage{graphics}
\usepackage{color,colortbl}
\usepackage{rotating}
\usepackage{fancybox}
\usepackage[table]{xcolor}
\usepackage{tikz}
\usetikzlibrary{matrix,arrows,shapes.geometric}

\usepackage[backend=biber, 
style=numeric,
maxbibnames=99,
sorting=none,
]{biblatex}

\begin{document}
\date{}
\title{
 Approximation of the Koopman operator via Bernstein polynomials
}
\maketitle
\vspace{-1cm}
\begin{center}
{\bf Rishikesh Yadav$^1$ and  Alexandre Mauroy$^2$}\\\vspace{.5cm}
{Namur Institute for Complex Systems (naXys) and Department of Mathematics,
University of Namur, 5000 Namur, Belgium}
\end{center}

\footnotetext[1]{rishikesh2506@gmail.com, rishikesh.yadav@unamur.be}\footnotetext[2]{alexandre.mauroy@unamur.be}

\begin{abstract}
The Koopman operator approach provides a powerful linear description of nonlinear dynamical systems in terms of the evolution of observables. While the operator is typically infinite-dimensional, it is crucial to develop finite-dimensional approximation methods and characterize the related approximation errors with upper bounds, preferably expressed in the uniform norm. In this paper, we depart from the traditional use of orthogonal projection or truncation, and propose a novel method based on Bernstein polynomial approximation. Considering a basis of Bernstein polynomials, we construct a matrix approximation of the Koopman operator in a computationally effective way. Building on results of approximation theory, we characterize the rates of convergence and the upper bounds of the error in various contexts including the cases of univariate and multivariate systems, and continuous and differentiable observables. The obtained bounds are expressed in the uniform norm in terms of the modulus of continuity of the observables. Finally, the method is extended to a data-driven setting through a proper change of coordinates. Numerical experiments show that the method is robust to noise and demonstrates good performance for trajectory prediction.
\end{abstract}



\section{Introduction}



Dynamical systems theory is of paramount importance in data science, since it is often possible to assume that time series are produced by an underlying dynamics. In the context of data-driven analysis of dynamical systems, the Koopman operator framework plays a crucial role \cite{BMM, MI}. Unlike traditional pointwise representations describing the evolution of the system state variables \textsl{in the state space}, the Koopman operator provides a powerful global perspective via the evolution of functions defined \textsl{over the state space}, also called observables \cite{Koopman}. Through this perspective, the partial knowledge gained through the data can be leveraged to obtain global information on the underlying dynamical system. Over the past years, the Koopman operator approach has also percolated to the related field of nonlinear control theory, providing an innovative body of linear data-driven techniques for system analysis and control design, e.g., 
stability analysis \cite{MAMI},
state observation \cite{Surana}, systems identification \cite{Mauroy2}, model predictive control \cite{Korda}, optimal control \cite{Goswami, Kaiser}, feedback stabilization \cite{Huang}, to list a few. See also \cite{Mauroy, BSH} for an overview.

The Koopman operator is typically infinite-dimensional, but can be approximated in a finite-dimensional linear subspace through truncation or orthogonal projection (see e.g. Extended Dynamic Mode Decomposition (EDMD) \cite{Williams}). This process yields a linear finite-dimensional description of the system, which is however vitiated by approximation errors. These errors should be characterized, not only through convergence rates that assess the consistency of the approximation as the subspace dimension increases, but also through upper bounds that can be taken into account by robust methods (e.g. robust control). To our knowledge, a first theoretical convergence analysis of finite-dimensional approximations of the Koopman operator has been carried out in \cite{Korda2}, with a focus on the EDMD method. Similarly, the work \cite{Mauroy2} investigated the convergence properties of approximations of the unbounded Koopman infinitesimal generator in the context of parameter estimation, but neither convergence rates nor error bounds were provided. On the other hand, convergence rates were derived in \cite{Kurdila} for Koopman operator approximations in the case of general systems and approximation spaces. Additionally, a convergence analysis was conducted in the specific case of interpolation projections in reproducing kernel Hilbert spaces (e.g. Sobolev spaces) \cite{Paruchuri, PBBK} and finite-element based approximations in $L^2$ spaces \cite{Zhanga}. Moreover, \cite{Peitz} derived probabilistic finite-data error bounds that were later used for feedback design \cite{SSWBA}, while \cite{Philipp2024} provided approximation error bounds for kernel-based EDMD. As shown by this body of work, there is a need for a theoretical examination of approximation error bounds. However, such bounds are often derived as a by-product of convergence analyses and cannot be explicitly computed. Also, they should be relevant to trajectory-oriented applications, and therefore expressed in appropriate norms such as the uniform norm (see e.g. the recent work \cite{Kuhne2024}).

As previously shown in the context of the Koopman operator (see e.g. \cite{Kurdila,Paruchuri}), approximation theory is a powerful framework to develop efficient finite-dimensional approximation methods and characterize the associated errors. Also, the use of Bernstein polynomials has appeared to be an efficient way to approximate continuous functions \cite{Bernstein}. For instance, Bernstein polynomials are used in one of the proofs of the celebrated Weierstrass Approximation Theorem \cite{korovkin1960linear}. Since then, they have been applied to various fields, such as neural networks \cite{FKYS}, optimal control \cite{CKWHP, Kielas}, control design \cite{Hamadneh2}, differential equations \cite{Doha1}, polynomial interpolation \cite{Marco}, computer-aided geometric design \cite{Winkel}. 
However, to our knowledge, they have not been leveraged in the context of the Koopman operator, apart from a minor use in \cite{MAMI}. In this paper, we fill this gap and propose a novel method to approximate the Koopman operator with Bernstein polynomials. The method has several advantages. Since it directly relies on approximation theory, it is complemented with convergence rates and upper bounds for the approximation error. These bounds are expressed in the uniform norm in terms of the known modulus of continuity of the observable functions (and possibly of their derivatives) and require the sole knowledge of the Lipschitz constant of the map describing the dynamics. Moreover, since the Bernstein operator is bounded in the uniform norm, the proposed approximation is robust to measurement noise, as demonstrated by numerical experiments. Also, a matrix approximation of the Koopman operator is constructed, which does not rely on the matrix pseudoinverse.  Finally, the framework is also amenable to a data-driven setting through the use of an appropriate change of variables, which maps randomly distributed data points onto a regular lattice. A uniform error bound is still provided in this case and numerical experiments suggest that the method also demonstrates good performance for trajectory prediction, in comparison to the EDMD method used with the same basis of functions.

This paper is organized as follows. Section 2 presents the main preliminaries, including the Bernstein polynomials, the Koopman operator framework, and fundamental definitions and properties. A matrix representation of the Koopman operator is derived in Section 3. Section 4 deals with the error analysis, providing error bounds and convergence rates for the approximation of the Koopman operator in one-dimensional and multidimensional cases. The error propagation under the iteration of the Koopman operator approximation is also considered. In Section 5, the approximation method is extended to the data-driven setting and compared to the EDMD method. Finally, conclusion and perspectives are given in Section 6.

\section{Preliminaries}

\subsection{Bernstein polynomials and Bernstein operator}
\label{sec:prelim_Bernstein}

A $n$th Bernstein polynomial, denoted below as $\mathcal{B}_n(f;x)$, is a polynomial of degree $n$ which approximates a continuous function $f$ defined on the interval $[0,1]$.
\begin{definition}
    For a function $f\in C([0,1])$, we define the $n$th Bernstein polynomial by
    \begin{equation*}
\mathcal{B}_n(f;x)=\sum\limits_{k=0}^n f\left(\frac{k}{n} \right) \, b_{n,k}(x), ~~x\in[0,1],~n\in\mathbb{N}
\end{equation*}
where $b_{n,k}(x)=\binom{n}{k}x^k(1-x)^{n-k}$ are the Bernstein basis polynomials of degree $n$. Moreover, we denote by $\mathcal{B}_n:C([0,1]) \to C([0,1])$ the (linear) Bernstein operator which maps $f$ to $\mathcal{B}_n(f;\cdot)$. 
\end{definition}

The above framework can be easily extended to multivariate functions. We have the following definition.
\begin{definition}
For a function $f\in C([0,1]^m)$ with $m\in \mathbb{N}^*$, we define the $m$-variate Bernstein polynomial by
\begin{equation}\label{devis}
\mathcal{B}_\mathbf{n}(f;\textbf{x})=\sum\limits_{{\substack{k_l=0 \\ l=1,\cdots,m}}
}^{n_l} f\left(\frac{k_1}{n_1},\frac{k_2}{n_2},\cdots, \frac{k_m}{n_m} \right) \prod_{l=1}^{m} b_{n_l,k_l}(x_l),\\ 
\end{equation}
with $\textbf{x}=(x_1, \cdots, x_m)\in[0,1]^m$ and $\mathbf{n}=(n_1, \cdots, n_m)\in\mathbb{N}^m$. Moreover, we denote by $\mathcal{B}_\mathbf{n}:C([0,1]^m)\to C([0,1]^m)$ the (linear) Bernstein operator which maps $f$ to $\mathcal{B}_\mathbf{n}(f;\cdot)$. 
\end{definition}

We note that the Bernstein operator is positive, i.e.
\begin{equation}
\label{eq:positivity}
    \mathcal{B}_\mathbf{n}(f;\textbf{x}) \geq 0
\end{equation}
if $f(\textbf{x})\geq 0$ for all $\textbf{x} \in [0,1]^m$.
In particular, this implies that
\begin{equation}
\label{eq:abs}
    \mathcal{B}_\mathbf{n}(|f|;\textbf{x}) \geq \mathcal{B}_\mathbf{n}(f;\textbf{x}) \quad \forall f\in C([0,1]^m).
\end{equation}
Also, the Bernstein operator is bounded with respect to the uniform norm, i.e., $\|\mathcal{B}_\mathbf{n}f\|_\infty\leq\|f\|_\infty$ with $\|f\|_\infty=\sup_{\textbf{x}\in[0,1]^m}|f(x)|$. Thanks to the linearity property, it also follows that $ \|\mathcal{B}_\mathbf{n}f-\mathcal{B}_\mathbf{n}\tilde{f}\|_\infty\leq\|f-\tilde{f}\|_\infty$, making the Bernstein approximation robust to (measurement) errors on the values of $f$.

According to Weierstrass theorem, we have that $\lim_{n \rightarrow \infty} \|\mathcal{B}_nf-f\|_{\infty}=0$ for all $f\in C([0,1])$ and, in the multivariate case,
\begin{equation*}
    \lim_{n_1 \rightarrow \infty} \cdots \lim_{n_m \rightarrow \infty} \|\mathcal{B}_\mathbf{n}f-f\|_{\infty}=0
\end{equation*}
for all $f\in C([0,1]^m)$. This convergence property along with the robustness to measurement errors motivates our use of Bernstein polynomials to approximate the Koopman operator.

\paragraph{Properties.} We finally provide a few basic properties of (multivariate) Bernstein polynomials, which we will use in our developments.

\begin{enumerate}
\item Bernstein basis polynomials form a partition of unity, i.e.
\begin{equation}
\label{eq:Bernstein_constant}
\mathcal{B}_\mathbf{n}(1;\textbf{x})=\sum_{l=1}^m \prod_{l=1}^{m} b_{n_l,k_l}(x_l) = 1.
\end{equation}
\item Bernstein operators preserve the linear polynomial, i.e., 
\begin{equation}
\label{eq:Bernstein_xl}
    \mathcal{B}_\mathbf{n}(u_l;\textbf{x})=x_l
\end{equation}
or equivalently we have

\begin{equation}
\label{eq:Bernstein_first_moment}
    \mathcal{B}_\mathbf{n}\left(\sum\limits_{l=1}^m(u_l-x_l);\textbf{x}\right)= \sum\limits_{{\substack{k_l=0 \\ l=1,\cdots,m}}}^{n_i} \sum\limits_{l=1}^m \left(\frac{k_l}{n_l} -x_l\right) \prod_{l=1}^{m}  b_{n_l,k_l}(x_l) = 0,~~u_l, x_l\in[0,1],
\end{equation}
which is related to the first central moment of the Bernstein polynomials.
\item The second central moment of the Bernstein polynomials satisfies
\begin{equation}
\label{eq:Bernstein_2nd_moment}
\mathcal{B}_\mathbf{n}\left(\sum\limits_{l=1}^m(u_l-x_l)^2;\textbf{x}\right)=\sum\limits_{{\substack{k_l=0 \\ l=1,\cdots,m}}}^{n_i} \sum\limits_{l=1}^m \left(\frac{k_l}{n_l} -x_l\right)^2\prod_{l=1}^{m}  b_{n_l,k_l}(x_l)  =  \sum\limits_{l=1}^m\frac{x_l(1-x_l)}{n_l}. \\
\end{equation}
\end{enumerate}

We can also observe that
\begin{equation}
\label{eq:partial1}
\mathcal{B}_\mathbf{n}(f;\textbf{x})=\mathcal{B}^{(1)}_{n_1}(\mathcal{B}^{(2)}_{n_2}(\dots(\mathcal{B}^{(m)}_{n_m}(f;x_m);x_{m-1})\dots);x_1)
\end{equation}
where $\mathcal{B}^{(l)}_{n_l}$ is the univariate Bernstein operator with respect to $l$th variable, that is
\begin{equation*}
\mathcal{B}^{(l)}_{n_l}(f;x_l) = \sum_{k=0}^{n_l} f\left(u_1,\dots,u_{l-1},\frac{k}{n_l},u_{l+1},\dots,u_m\right) b_{n_l,k}(x_l), 
\end{equation*}
with $u_l,x_l\in[0,1]$, $l=1,2,\dots,m.$
It is clear that the operators can commute in \eqref{eq:partial1}, so that we can rewrite
\begin{equation}
\label{eq:partial2}
\mathcal{B}_\mathbf{n}(f;\textbf{x}) = \overline{\mathcal{B}}^{(-l)}_{\mathbf{n}}(\mathcal{B}^{(l)}_{n_l}(f;x_l);x_1,\dots,x_{l-1},x_{l+1},\dots,x_m) 
\end{equation}
with $\overline{\mathcal{B}}^{(-l)}_{\mathbf{n}} = \mathcal{B}^{(1)}_{n_1} \cdots \mathcal{B}^{(l-1)}_{n_{l-1}} \mathcal{B}^{(l+1)}_{n_{l+1}} \cdots \mathcal{B}^{(m)}_{n_m}$. Then, we have the following lemma.
\begin{lemma}
\label{lem:partial}
    For all $f\in C([0,1]^m)$,
    \begin{equation*}
    \begin{split}
        &\left| \mathcal{B}_\mathbf{n}(f;\textbf{x}) - f(\textbf{x})\right| \\
        & \quad \leq \sum_{l=1}^m \overline{\mathcal{B}}^{(-l)}_{\mathbf{n}}\left(\left| \mathcal{B}^{(l)}_{n_l}(f(x_1,\dots,x_{l-1},u_l,\dots,u_m) - f(x_1,\dots,x_l,u_{l+1},\dots,u_m);x_l)\right|;x_1,\dots,x_{l-1},x_{l+1},\dots,x_m\right),
        \end{split}
    \end{equation*}
    with $u_l,x_l\in[0,1]$, $l=1,2,\dots,m.$
\end{lemma}
\begin{proof}
For $f\in C([0,1]^m)$ and $\textbf{x}\in [0,1]^m$, it follows from \eqref{eq:Bernstein_constant} and the triangular inequality that
\begin{equation*}
\begin{split}
\left| \mathcal{B}_\mathbf{n}(f;\textbf{x}) - f(\textbf{x})\right| & = \left| \mathcal{B}_\mathbf{n}(f(\textbf{u}) - f(\textbf{x});\textbf{x}) \right| \\
&  \leq \left| \mathcal{B}_\mathbf{n}(f(u_1,\dots,u_m) - f(x_1,u_2,\dots,u_m);\textbf{x})\right| \\
& \quad + \left| \mathcal{B}_\mathbf{n}(f(x_1,u_2,\dots,u_m) - f(x_1,x_2,u_3,\dots,u_m);\textbf{x})\right| \\
& \quad + \cdots + \left| \mathcal{B}_\mathbf{n}(f(x_1,\dots,x_{m-1},u_m) - f(x_1,\dots,x_m);\textbf{x})\right|. \\
\end{split}
\end{equation*} 
Using \eqref{eq:abs} and \eqref{eq:partial2}, we obtain
\begin{equation*}
\begin{split}
\left| \mathcal{B}_\mathbf{n}(f;\textbf{x}) - f(\textbf{x})\right| 
& \leq  \overline{\mathcal{B}}^{(-1)}_{\mathbf{n}}\left(\left| \mathcal{B}^{(1)}_{n_1}(f(u_1,\dots,u_m) - f(x_1,u_2,\dots,u_m);x_1)\right|;x_2,\dots,x_m\right) \\
& \quad + \overline{\mathcal{B}}^{(-2)}_{\mathbf{n}}\left(\left| \mathcal{B}^{(2)}_{n_2}(f(x_1,u_2,\dots,u_m) - f(x_1,x_2,u_3,\dots,u_m);x_2)\right|;x_1,x_3\dots,x_m\right) \\
& \quad + \cdots + \overline{\mathcal{B}}^{(-m)}_{\mathbf{n}}\left(\left| \mathcal{B}^{(m)}_{n_m}(f(x_1,\dots,x_{m-1},u_m) - f(x_1,\dots,x_m);x_m)\right|;x_1,\dots,x_{m-1}\right). \\
\end{split}
\end{equation*} 
\end{proof}

\subsection{Koopman operator framework}

Let us consider a continuous map $\phi:D\to D$ that describes a dynamical system in a state space $D \in \mathbb{R}^m$. This map can provide the orbits of a discrete-time system, or can be extracted from the flow $(\phi^t)_{t\geq 0}:D\to D$ generated by a continuous-time system
\begin{equation}\label{dyna}
\dot{\textbf{x}}=F(\textbf{x}),~~\textbf{x}\in D
\end{equation}
where $F:D\to\mathbb{R}^m$ is the vector field. Next, suppose we are given a space $\mathsf{F}$ of functions (or observables) $f:D\to\mathbb{C}$. We have the following definition.
\begin{definition}
 The Koopman (or composition) operator $\mathcal{K}:\mathsf{F}\to\mathsf{F}$ associated with the (continuous) map $\phi$ is defined by
\begin{equation}
\mathcal{K} f=f\circ\phi.
\end{equation}   
\end{definition}
We note that the operator $\mathcal{K}$ is linear, even when the map $\phi$ is nonlinear.

From this point on, we will assume that $D=[0,1]^m$ and consider the space of functions is $\mathsf{F}=C([0,1]^m)$ (equipped with the supremum norm $\|\cdot\|_\infty$), as well as the subspace $\mathsf{F}_\mathbf{n} \subset \mathsf{F}$ of $m$-dimensional polynomials with degree less or equal to $n_l$ in $x_l$, with $l\in\{1,\dots,m\}$. In this case, the Koopman operator can be approximated with the Bernstein operator $~\mathcal{B}_\mathbf{n}:\mathsf{F} \to \mathsf{F}_\mathbf{n}$ according to
 \begin{eqnarray*}
\overline{\mathcal{K}} &:& \mathsf{F}\to \mathsf{F}_\mathbf{n}\\
&& f \mapsto \overline{\mathcal{K}} f  \triangleq ~\mathcal{B}_\mathbf{n} \mathcal{K} f.
\end{eqnarray*}

It follows that we have
\begin{equation}\label{pr}
(\overline{\mathcal{K}} f)(\textbf{x})=\mathcal{B}_\mathbf{n}(\mathcal{K}f;\textbf{x}) =\sum\limits_{{\substack{k_l=0 \\ l=1,\cdots,m}}
}^{n_l} f \circ \phi \left(\frac{k_1}{n_1},\frac{k_2}{n_2},\cdots, \frac{k_m}{n_m} \right) \left(\prod_{l=1}^{m}  b_{n_l,k_l}(x_l)\right), 
\end{equation}
where $\textbf{x}\in[0,1]^m,~\mathbf{n}\in\mathbb{N}^n.$
In Section \ref{sec:matrix}, this approximation will allow us to derive a matrix representation of the operator, which can be used, among other applications, for trajectory prediction.

Since $\mathcal{K}f=f \circ \phi$ is continuous, it follows from Weierstrass theorem that
\begin{equation*}
    \lim_{n_1 \rightarrow \infty} \cdots \lim_{n_m \rightarrow \infty} \|\mathcal{B}_\mathbf{n}\mathcal{K}f-\mathcal{K}f\|_{\infty}=0
\end{equation*}
for $f\in \mathsf{F}$, so that we have uniform convergence of $\overline{\mathcal{K}}f$ to $\mathcal{K}f$. In Section \ref{sec:bounds}, we will investigate the convergence rate and the bounds on this approximation error $\|\overline{\mathcal{K}}f-\mathcal{K}f\|_{\infty}$.

\subsection{Modulus of continuity and Lipschitz continuity}

The theoretical bounds on the approximation error depend on regularity properties of the functions. These properties are related to the following basic concepts, which we recall here for the sake of completeness.

\paragraph{Modulus of continuity}
\begin{definition}
The full modulus of continuity of $f\in C[0,1]^m$ is the function
\begin{eqnarray}\label{def1}
\Omega_f(\delta)=\underset{\|\textbf{x}-\textbf{y}\|\leq\delta}\max \{|f(x_1,x_2,\cdots,x_m)-f(y_1,y_2,\cdots,y_m)|,\, \textbf{x},\textbf{y}\in [0,1]^m\},
\end{eqnarray}
where $\delta$ is a positive number.
\end{definition}

\begin{definition}\label{def2}
The partial modulus of continuity of $f\in C[0,1]^m$ with respect to $x_l$, $l=1,\cdots, m$, is the function
\begin{equation*}
\Omega^{l}_f(\delta)=\underset{{\substack{x_j \in[0,1] \\ j\neq l }}}\max ~~\underset{{\substack{|x_l-y_l| \leq \delta\\ x_l, y_l\in [0,1] }}}\max \left|f(x_1,\cdots, x_{l-1}, x_{l}, x_{l+1}, \cdots, x_m)-f(x_1,\cdots, x_{l-1}, y_{l}, x_{l+1}, \cdots, x_m)\right|.
\end{equation*}
\end{definition}

We note that the (full and partial) modulus of continuity satisfies the following useful properties: 

\begin{enumerate}
    \item \label{pro1}$\underset{\delta\to 0}\lim \Omega_f(\delta)= 0, \qquad \underset{\delta\to 0}\lim \Omega^{(l)}_f(\delta)= 0,$
    \item \label{pro2} For any $\delta>0$ and $\nu>0$, 
\begin{equation}
\label{eq:prop_modulus_nu}
    \Omega_f(\nu\delta)\leq (1+\nu)\Omega_f(\delta), \qquad \Omega^{(l)}_f(\nu\delta)\leq (1+\nu)\Omega^{(l)}_f(\delta),
\end{equation}
\item \label{pro3} For any $\delta>0$,
\begin{equation}
\label{eq:prop_modulus}
    |f(\textbf{y})-f(\textbf{x})| \leq \Omega_f(\delta)\left(\frac{\|\textbf{y}-\textbf{x}\|}{\delta}+1 \right), \qquad |f(\textbf{y})-f(\textbf{x})| \leq \Omega^{(l)}_f(\delta)\left(\frac{\|\textbf{y}-\textbf{x}\|}{\delta}+1 \right).
\end{equation}
\end{enumerate}

In the particular case of univariate functions $f$, we will denote by $\omega_f$ the modulus of continuity
\begin{equation*}
\omega_f(\delta)=\Omega_f(\delta)=\sup\left\{|f(y)-f(x)|:|y-x|\leq\delta,~x,y\in [0,1] \right\}.
\end{equation*}
Note also that we can generalize the modulus of continuity for vector-valued functions $f:[0,1]^m \to \mathbb{R}^p$:
\begin{equation}
    \Omega_{f}(\delta)=\max_{\|\textbf{x}-\textbf{y}\|\leq \delta} \{\|f(\textbf{x})-f(\textbf{y})\|,\, \textbf{x},\textbf{y}\in [0,1]^m\}.
\end{equation}
In this case, we have $\Omega_{f}(\delta) \leq \sum_{l=1}^p \Omega_{f_l}(\delta)$.

\paragraph{Lipschitz continuity}
\begin{definition}
A function $f: \mathbb{R}^n\to\mathbb{R}^m$ is Lipschitz continuous if there exists a constant $L_f$ (denoted hereafter as a Lipschitz constant) such that, for all 
\begin{equation}
\|f(\textbf{x})-f(\textbf{y})\|\leq L_f\|{\textbf{x}}-{\textbf{y}}\|, ~\forall {\textbf{x}}, {\textbf{y}}\in\mathbb{R}^n.
\end{equation}
\end{definition}
Lipschitz continuity also implies partial Lipschitz continuity, in the sense that, for all $l=1,2,\cdots,m$, there exists a constant $L_f^{(l)}$ such that
    \begin{equation*}
\|f(x_1,\cdots, x_{l-1},x_l,x_{l+1}, \cdots, x_m)-f(x_1,\cdots,x_{l-1},y_l,x_{l+1}, \cdots, x_m)\|\leq L_f^{(l)} |x_l-y_l|, ~\forall {\textbf{x}}, {\textbf{y}}\in\mathbb{R}^m.
\end{equation*}

In fact, it can be seen that $L_f^{(l)} \leq L_f \leq \sqrt{\sum_{l=1}^m (L_f^{(l)})^2}$ if $L_f$ and $L_f^{(l)}$ are the smallest Lipschitz constants.

We will use the following result, which connects the modulus of continuity with the Koopman operator.
\begin{lemma}
\label{lem:modulus}
    Let $\mathcal{K}$ be the Koopman operator associated with the Lipschitz continuous map $\phi:[0,1]^m \to [0,1]^m$ (with Lipschitz constant $L_\phi$ and partial Lipschitz constants $L_\phi^{(l)}$). Then, for all $f\in C([0,1]^m)$ and $\delta>0$, we have
    \begin{equation*}
        \Omega_{\mathcal{K}f}(\delta) \leq \left. \Omega_{f}\left(L_{\phi} \delta \right)\right|_{\phi([0,1]^m)}   \quad \textrm{and} \quad \Omega^{(l)}_{\mathcal{K}f}(\delta) \leq \left. \Omega_{f}\left(L_\phi^{(l)} \delta \right)\right|_{\phi([0,1]^m)}, \quad l=1,\dots,m.
    \end{equation*}
\end{lemma}
\begin{proof}
We have
\begin{equation*}
\Omega_{\mathcal{K}f}\left(\delta\right) = \sup\left\{|f(\phi(\textbf{x}))-f(\phi(\textbf{y}))|:~\|\textbf{x}-\textbf{y}\|\leq \delta,~~\textbf{x},\textbf{y}\in[0,1]^m\right\}.
\end{equation*}
It is clear that the inequality $\|\textbf{x}-\textbf{y}\|\leq \delta$ and the Lipschitz continuity of $\phi$ imply that
\begin{equation*}
    \|\phi(\textbf{x})-\phi(\textbf{y})\|\leq L_{\phi}\|\textbf{x}-\textbf{y}\| \leq  L_{\phi} \delta
\end{equation*}
and it follows that
\begin{eqnarray*}
    \Omega_{\mathcal{K}f}\left(\delta\right) & \leq & \sup\left\{|f(\phi(\textbf{x}))-f(\phi(\textbf{y}))|:~||\phi(\textbf{x})-\phi(\textbf{y})||\leq L_{\phi} \delta ,~~\textbf{x},\textbf{y}\in[0,1]^m\right\} \\
    & = & \left. \Omega_{f}\left(L_{\phi} \delta \right)\right|_{\phi([0,1]^m)}.
\end{eqnarray*}
Similarly, we have
\begin{equation*}
\Omega^{l}_{\mathcal{K}f}\left(\delta\right)=\underset{{\substack{x_j \in[0,1] \\ j\neq l }}}\max ~~\underset{{\substack{|x_l-y_l| \leq \delta\\ x_l, y_l\in [0,1] }}}\max |f(\phi((x_1,\cdots, x_{l-1}, x_{l}, x_{l+1}, \cdots, x_m))-f(\phi((x_1,\cdots, x_{l-1}, y_{l}, x_{l+1}, \cdots, x_m))|.
\end{equation*}
Since
\begin{equation*}
\|\phi(x_1,\cdots, x_{l-1},x_l,x_{l+1}, \cdots, x_m)-\phi(x_1,\cdots,x_{l-1},y_l,x_{l+1}, \cdots, x_m)\| \leq L_\phi^{(l)} |x_l-y_l|
\end{equation*}
the inequality $|x_l-y_l|<\delta$ implies
\begin{equation*}
\|\phi(x_1,\cdots, x_{l-1},x_l,x_{l+1}, \cdots, x_m)-\phi(x_1,\cdots,x_{l-1},y_l,x_{l+1}, \cdots, x_m)\| \leq L_\phi^{(l)} \delta,
\end{equation*}
so that
\begin{equation*}
    \Omega^{l}_{\mathcal{K}f}\left(\delta\right)\leq \left.\Omega_{f}\left(L_\phi^{(l)} \delta \right)\right\vert_{\phi([0,1]^m)}.
\end{equation*}
\end{proof}

Finally, the modulus of continuity can be used to characterize the error on the Bernstein approximation of the Koopman operator when the value of the map $\phi$ is measured with some error $\Delta$. The following simple result shows that the approximation is robust to such measurement error.
\begin{lemma}
\label{lem:meas_error}
    Let $\mathcal{K}$ and $\mathcal{K}_\Delta$ be the Koopman operators associated with the continuous maps $\phi:[0,1]^m \to [0,1]^m$ and $\phi_\Delta:[0,1]^m \to [0,1]^m$ with $\phi_\Delta=\phi+\Delta$. Then, for all $f\in C([0,1]^m)$, we have
    \begin{equation*}
     \|\mathcal{B}_\mathbf{n}\mathcal{K} f-\mathcal{B}_\mathbf{n}\mathcal{K}_\Delta f\|_\infty \leq \Omega_f(\|\Delta\|_\infty).
    \end{equation*}
\end{lemma}
\begin{proof}
    We have that
    \begin{equation*}
        \|\mathcal{K} f-\mathcal{K}_\Delta f \|_\infty = \sup_{\mathbf{x}\in [0,1]^m} |f(\phi(\mathbf{x})) - f(\phi_\Delta(\mathbf{x}))| \leq \Omega_f(\|\Delta\|_\infty)
    \end{equation*}
    and the result follows from the fact that $\|\mathcal{B}_\mathbf{n}\mathcal{K} f-\mathcal{B}_\mathbf{n}\mathcal{K}_\Delta f\|_\infty \leq \|\mathcal{K} f-\mathcal{K}_\Delta f \|_\infty$.
\end{proof}

\section{Matrix representation of the Koopman operator}
\label{sec:matrix}

In this section, we will construct the matrix representation $\mathbf{K}_\mathcal{B}$ of the finite-dimensional Bernstein approximation $\overline{\mathcal{K}}=\mathcal{B}_\mathbf{n}\mathcal{K}$ of the Koopman operator $\mathcal{K}$, restricted to the subspace $\mathsf{F}_\mathbf{n}$ of polynomials of degree $\mathbf{n}$. This so-called Koopman matrix $\mathbf{K}_\mathcal{B}$ is such that
\begin{equation*}
    \mathcal{B}_\mathbf{n}(\mathcal{K} B;\textbf{x}) = \overline{\mathcal{K}} B(\textbf{x}) = \mathbf{K}_\mathcal{B} \, B(\textbf{x}),
\end{equation*}
with the vector of Bernstein polynomials
\begin{equation*}
B(\textbf{x})=B^{(1)}(x_1)\otimes \cdots\otimes B^{(m)}(x_m) = \begin{bmatrix}
    b_{n_1,0}(x_1)\\
    b_{n_1,1}(x_1)\\
\vdots\\
b_{n_1,n_1}(x_1)
\end{bmatrix} \otimes \cdots\otimes \begin{bmatrix}
b_{n_m,0}(x_m)\\
b_{n_m,1}(x_m)\\
\vdots\\
b_{n_m,n_m}(x_m)
\end{bmatrix}
\end{equation*}
and where $\mathcal{K}B$ is a vector obtained by applying $\mathcal{K}$ on each component of $B$. We note that the components of $B$ are $m$-dimensional Bernstein polynomials $B_j=\prod_{l=1}^{m} b_{n_l,\alpha_l(j)}$, with the map $\alpha:\mathbb{N} \to \mathbb{N}^m$ that refers to the lexicographic order (given the definition of the Kronecker product). It follows that the matrix $\mathbf{K}_\mathcal{B}$ is the representation of the approximation of the Koopman operator in the Bernstein polynomial basis of $\mathsf{F}_\mathbf{n}$. Alternatively, the matrix approximation of the Koopman operator can also be represented in the basis of monomials. Since
\begin{equation*}
b_{n_l,k}(x_l)=\sum\limits_{j=k}^{n_l}(-1)^{j-k}{n_l\choose j}{j\choose k}x_l^j, \qquad l=1,\dots,m
\end{equation*}
we have $B^{(l)}(x_l) = \mathbf{C}^{(l)} X^{(l)}(x_l)$, with the vector of monomials
\begin{equation*}
X^{(l)}(x_l)=\begin{bmatrix}
    1\\
    x_l\\
\vdots\\
x_l^{n_l}
\end{bmatrix}, \qquad l=1,\dots,m
\end{equation*}
and the matrix $\mathbf{C}^{(l)} \in \mathbb{R}^{(n+1)\times(n+1)}$ with entries
\begin{equation*}
    \mathbf{C}^{(l)}_{k+1,j+1} = (-1)^{j-k}{n_l\choose j}{j\choose k}, \qquad l=1,\dots,m.
\end{equation*}

Denoting $X(\textbf{x})=X^{(1)} \otimes \cdots \otimes X^{(m)}$ and $\mathbf{C}=\mathbf{C}^{(1)} \otimes \cdots \otimes \mathbf{C}^{(m)}$, we obtain
\begin{equation}
\label{eq:Bernstein_monomials}
B(\textbf{x}) = \mathbf{C} X(\textbf{x})
\end{equation}
and we can define the Koopman matrix representation in the monomial basis:
\begin{equation*}
    \mathbf{K}^X_\mathcal{B} = \mathbf{C}^{-1} \mathbf{K}_\mathcal{B} \mathbf{C}.
\end{equation*}
We verify that we have
\begin{equation*}
\begin{split}
    \mathcal{B}_\mathbf{n}(\mathcal{K} X;\textbf{x}) = \mathcal{B}_\mathbf{n}(\mathcal{K} \mathbf{C}^{-1} B;\textbf{x}) & = \mathbf{C}^{-1} \mathcal{B}_\mathbf{n}(\mathcal{K} B;\textbf{x}) \\
    & = \mathbf{C}^{-1} \mathbf{K}_\mathcal{B} B(\textbf{x}) = \mathbf{C}^{-1} \mathbf{K}_\mathcal{B} \mathbf{C} X(\textbf{x}) = \mathbf{K}^X_\mathcal{B} X(\textbf{x}).
    \end{split}
\end{equation*}

We can now derive the expression of the matrix $\mathbf{K}_\mathcal{B}$. We first note that
\begin{equation*}
    [\mathbf{K}_\mathcal{B} B(\textbf{x})]_i = \mathcal{B}_\mathbf{n}(\mathcal{K} B_i;\textbf{x}) = \mathcal{B}_\mathbf{n} (B_i\circ \phi;\textbf{x}) = \sum\limits_{{\substack{k_l=0 \\ l=1,\cdots,m}}
}^{n_l} B_i \circ \phi \left(\frac{k_1}{n_1},\cdots, \frac{k_m}{n_m} \right) \prod_{l=1}^{m} b_{n_l,k_l}(x_l)
\end{equation*}
where $[\cdot]_{i}$ denotes the $i$th component of a vector. Equivalently, we can write
\begin{equation*}
      [\mathbf{K}_\mathcal{B} B(\textbf{x})]_i = \sum\limits_{j=1}^{N} B_i \circ \phi \left(\hat{x}_j\right) B_j(\textbf{x})
\end{equation*}
with $\hat{x}_j = (\alpha_1(j)/n_1,\cdots, \alpha_m(j)/n_m)$ and $N=\Pi_{j=1}^m (n_j+1)$. Then it is clear that the $i$th row of $\mathbf{K}_\mathcal{B}$ has entries of the form $B_i(\phi(\hat{x}_j))$, i.e.
\begin{equation*}
    \mathbf{K}_\mathcal{B} = [B(\phi(\hat{x}_1)) \, B(\phi(\hat{x}_2)) \, \cdots \, B(\phi(\hat{x}_{N})))].
\end{equation*}
Using \eqref{eq:Bernstein_monomials}, we can write
\begin{equation*}
    \mathbf{K}_\mathcal{B} = \mathbf{C} \mathbf{U} \, 
\end{equation*}
with the matrix
\begin{equation*}
    \mathbf{U} = [X(\phi(\hat{x}_1)) \, X(\phi(\hat{x}_2)) \, \cdots \, X(\phi(\hat{x}_{N})))].
\end{equation*}
Moreover, we also have
\begin{equation*}
    \mathbf{K}_\mathcal{B}^X = \mathbf{U} \mathbf{C}.
\end{equation*}
Note that both matrices $\mathbf{K}_\mathcal{B}$ and $\mathbf{K}_\mathcal{B}^X$ can be computed efficiently through a mere matrix multiplication.

Finally, if we denote by $\gamma_j$ the index associated with the monomial $x_j$, i.e. $[X(\textbf{x})]_{\gamma_j}=x_j$ (note that we should have $\gamma_j=1+\prod_{l=j+1}^m (n_l+1)$ for $j<m$ and $\gamma_m=2$), we can obtain
\begin{equation}
    \label{eq:prediction}
    \phi_j(\textbf{x}) = \mathcal{K} x_j \approx \overline{\mathcal{K}} x_j = [(\mathbf{K}^X_\mathcal{B}) X(\textbf{x})]_{\gamma_j} = [(\mathbf{U} \mathbf{C}) X(\textbf{x})]_{\gamma_j}.
\end{equation}
Iterating the Koopman matrix $\mathbf{K}_\mathcal{B}$ or $\mathbf{K}^X_\mathcal{B}$ therefore provides a linear approximation of the system trajectory.

The following example leverages the matrix Bernstein approximation in the context of trajectory prediction. Although the Koopman operator framework can be used for various purposes that are not limited to prediction (e.g. analysis, control, etc.), the latter is a natural performance criterion to assess the quality of the approximation.

\begin{example}
Consider the Van der Pol dynamics
\begin{eqnarray*}
\dot{x}_1&=&x_2\\
\dot{x}_2&=&0.5(1-x_1^2)x_2-x_1
\end{eqnarray*}
and the flow map $\phi$ generated at time $t=0.3$, which is rescaled from $[-3,3]^2$ to $[0,1]^2$. The Bernstein approximation of the Koopman operator is computed for $n=n_1=n_2\in\{10,20,25\}$. In Figure \ref{ven1}, the true trajectory starting from the initial condition $x_0=(0.4,0.)$ is compared with an estimated trajectory computed with the linear approximation \eqref{eq:prediction}, i.e. $\phi_j^{kt}(x_0) \approx [ (\mathbf{K}^{X}_\mathcal{B})^k X(x_0)]_{\gamma_j}$. It is observed that increasing the degree of the polynomials increases the short-term prediction accuracy (see Table \ref{tab:short_term_vdp}). For large degrees, the Bernstein approximation provides very good prediction of a few iterations ahead, but finally diverges due to the matrix instability. This illustrates a trade-off between short-term and long-term prediction, which requires a proper tuning of the polynomial degree $n$. Note that, alternatively, the values of the monomials could be computed at each iteration, i.e. $\phi_j^{(k+1)t} \approx [(\mathbf{K}_\mathcal{B}^X X) X(\phi^{kt})]_{\gamma_j}$. This yields better prediction (not shown) but at the cost of a nonlinear predictor that does not leverage the linear Koopman operator framework. Finally, we investigate the performance of the approximation in the context of measurement noise. To do so, we use the values $\phi(\hat{x}_j)+\Delta$, where $\Delta$ is a Gaussian random vector with zero mean and component-wise standard deviation $\sigma \in \{0.001, 0.01, 0.1\}$. We observe in Table \ref{tab:noise} that the Bernstein approximation is robust to noise and is in particular not affected by low noise level, in agreement with the result of Lemma \ref{lem:meas_error}.

\begin{figure}[h!]
     \centering  
 \includegraphics[width=0.5\textwidth]{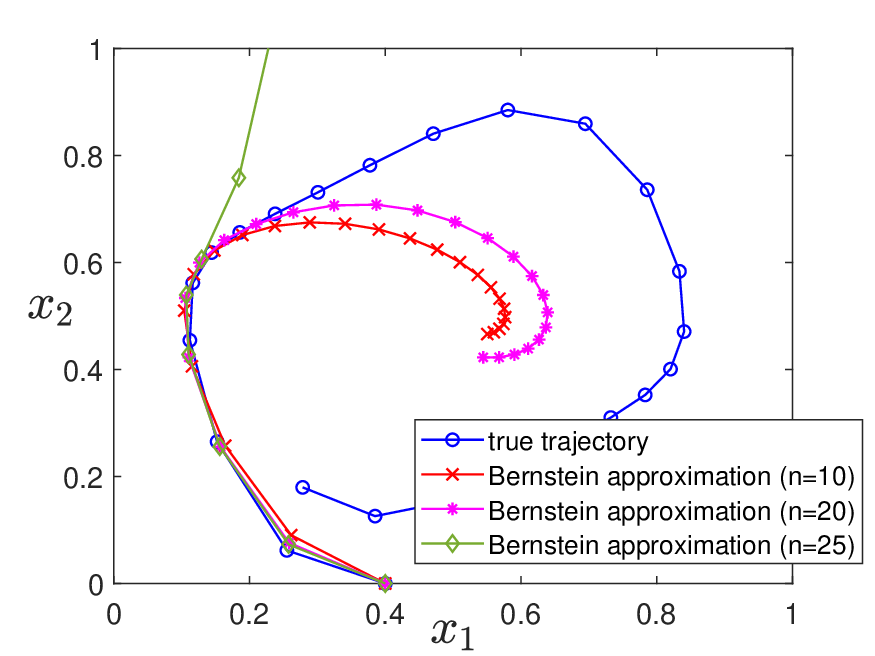} 
    \caption[]{Predicted trajectories of the Van der Pol system obtained with the Bernstein approximation of the Koopman operator, for different degrees of the polynomials.}
    \label{ven1}
\end{figure}

\begin{table}
    \centering
     
    \begin{tabular}{c|cccccc}
            $n$ & $k=1$  & $k=2$ & $k=3$ & $k=4$ & $k=5$ & $k=6$ \\
         \hline
         $10$ & 0.0290 &  0.0141  &  0.0487  &  0.0533  &  0.0483 &  0.0503 \\
         $20$ & 0.0144 &   0.0112  &  0.0319  &  0.0303  &  0.0260  &  0.0271 \\
         $25$ & 0.0115  &  0.0100  &  0.0266  &  0.0245  &  0.0191  &  0.1021
    \end{tabular}
    \caption{Short-term prediction error for the Van der Pol system.}
    \label{tab:short_term_vdp}
\end{table}

\begin{table}
    \centering
  \pagebreak   
    \begin{tabular}{c|cccc}
            $n$ & $\sigma=0$  & $\sigma=0.001$ & $\sigma=0.01$ & $\sigma=0.1$ \\
         \hline
         $10$ & 0.0290 &  0.0291  &  0.0295  &  0.0627 \\
         $20$ & 0.0144 &  0.0143  &  0.0152  &  0.0439 \\
         $25$ & 0.0115 &  0.0114 & 0.0121 & 0.0434
    \end{tabular}
    \caption{One-step prediction error for the Van der Pol system, in the case of noisy measurements (Gaussian random noise with standard deviation $\sigma$). The error is averaged over $50$ simulations.}
    \label{tab:noise}
\end{table}

\end{example}

\pagebreak
\section{Bounds on the approximation error}
\label{sec:bounds}

In this section, we estimate approximation error bounds by assuming that the observable functions are either continuous or continuous differentiable. Convergence rates of the approximation error are also obtained. We first consider the case of one-dimensional systems, and then extend the results to the multivariate case.

\subsection{One-dimensional case}

We first have the following simple result, which provides a uniform bound on the Bernstein approximation of the Koopman operator in the case of one-dimensional systems described by a Lipschitz continuous map. The result is derived from a well-known result on Bernstein approximation errors. \\

\begin{theorem}\label{Th1}
Let $\phi:[0,1] \to [0,1]$ be a Lipschitz continuous map (with Lipschitz constant $L_\phi$). Then, for any $f\in C([0,1])$, an upper bound on the error in approximating $\mathcal{K}f$ by $\mathcal{B}_n\mathcal{K}f$ is given by
\begin{equation}
\label{eq:result1}
\|\mathcal{B}_n\mathcal{K}f-\mathcal{K}f\|_\infty\leq \frac{3}{2}\omega_{f}\left(\frac{L_{\phi}}{\sqrt{n}}\right)\bigg|_{\left[\phi(0),\phi(1)\right]}.
\end{equation}
\end{theorem}

\begin{proof}
Since $\phi$ is continuous in $[0,1]$, so is the function $\mathcal{K}f=f \circ \phi$. Then, it is well-known that a bound on the Bernstein approximation error for $\mathcal{K}f$ is given by (see e.g. \cite{Carothers})
\begin{equation}\label{in2}
\|\mathcal{B}_n\mathcal{K}f-\mathcal{K}f\|_{\infty}\leq \frac{3}{2}\omega_{\mathcal{K}f}\left(\frac{1}{\sqrt{n}}\right).
\end{equation}
Moreover, it follows from Lemma \ref{lem:modulus} that
\begin{equation*}
    \omega_{\mathcal{K}f}\left(\frac{1}{\sqrt{n}}\right) \leq  \omega_{f}\left(\frac{L_{\phi}}{\sqrt{n}}\right)\bigg|_{\left[\phi(0),\phi(1)\right]},
\end{equation*}
which completes the proof.
\end{proof}

Note that the modulus of continuity of $f$ evaluated over $[\phi(0),\phi(1)] \subseteq [0,1]$ can be bounded by the modulus of continuity over $[0,1]$. Moreover, in the case $L_\phi/\sqrt{n}>1$, the value of the modulus of continuity $\omega_f(L_\phi/\sqrt{n})$ should obviously be replaced by $\omega_f(1)$ in \eqref{eq:result1}.

According to the above result, when the function $f$ is assumed to be continuous, the convergence rate of the approximation error is of the order of $n^{-1/2}$, which can be considered as slow. In order to obtain a faster rate of convergence, we now assume that the function $f$ and the map $\phi$ are continuously differentiable.

\begin{theorem}\label{Th2}
Let $\phi:[0,1] \to [0,1]$ be a Lipschitz continuous map (with Lipschitz constant $L_{\phi}$) and whose derivative is also Lipschitz continuous (with Lipschitz constant $L_{\phi'}$). Then, for any $f\in C^1([0,1])$, an upper bound on the error in approximating $\mathcal{K}f$ by $\mathcal{B}_n\mathcal{K}f$ is given by
\begin{eqnarray*}
\|\mathcal{B}_n\mathcal{K}f-\mathcal{K}f\|_\infty\leq \frac{1}{\sqrt{n}} \left( L_{\phi} \, \omega_{f'}\left(\frac{L_{\phi}}{2\sqrt{n}}\right)\bigg|_{[\phi(0),\phi(1)]}     +\underset{0\leq x\leq1}\sup |f'(\phi(x))|\, \frac{L_{\phi'}}{2\sqrt{n}}  \right).
\end{eqnarray*}
\end{theorem}

\begin{proof}
Let $x$ be a fixed point in $[0,1]$. Then, for $f\in C^1[0,1]$, one can write
\begin{equation*}
f(u)-f(x)=(u-x)f'(x)+\int\limits_x^u (f'(v)-f'(x))~dv.
\end{equation*}
Applying the operator $\mathcal{B}_n$ on both sides of the above expression, we obtain
\begin{eqnarray}
\mathcal{B}_n(f(u);x)-f(x) \mathcal{B}_n(1;x)&=& \left(\mathcal{B}_n(u;x)-x\mathcal{B}_n(1;x) \right) f'(x) +\mathcal{B}_n\left(\int\limits_x^u (f'(v)-f'(x))~dv;x\right)\nonumber\\
\mathcal{B}_n(f(u);x)-f(x)&=& \mathcal{B}_n\left(\int\limits_x^u (f'(v)-f'(x))~dv;x\right),
\end{eqnarray}
where we used \eqref{eq:Bernstein_constant} and \eqref{eq:Bernstein_xl}. By using the property \eqref{eq:prop_modulus} of the modulus of continuity, we obtain
\begin{eqnarray*}
\left| \left(\int\limits_x^u (f'(v)-f'(x))~dv\right)\right|&\leq & \int\limits_x^u |f'(v)-f'(x)|~dv \\
&\leq & \int\limits_x^u \omega_{f'}(\delta) \left( \frac{|v-x|}{\delta}+1\right)~dv \\
&\leq &\omega_{f'}(\delta)\left(\frac{(u-x)^2}{\delta} +|u-x|\right)
\end{eqnarray*}
for any $\delta>0$. Then, it follows from \eqref{eq:positivity} and \eqref{eq:abs} that
$$|\mathcal{B}_n(f(u);x)-f(x)| \leq \mathcal{B}_n\left(\left| \left(\int\limits_x^u (f'(v)-f'(x))~dv\right)\right| ;x\right) \leq \omega_{f'}(\delta) \, \mathcal{B}_n \left(\frac{(u-x)^2}{\delta} +|u-x| \right). $$
Using \eqref{eq:Bernstein_constant}, \eqref{eq:Bernstein_2nd_moment}, and the Cauchy-Schwartz inequality, we have
\begin{eqnarray*}
|\mathcal{B}_n(f(u);x)-f(x)| &\leq & \omega_{f'}(\delta)\left(\frac{x(1-x)}{n\delta} +\mathcal{B}_n(|u-x|;x)\right)\\
&= & \omega_{f'}(\delta)\left(\frac{x(1-x)}{n\delta} +\left(\sum\limits_{k=0}^n (b_{n,k}(x))^{\frac{1}{2}}\left|\left(\frac{k}{n} -x\right)^2\right|^{\frac{1}{2}}\cdot (b_{n,k}(x))^{\frac{1}{2}}\right) \right)\\
&\leq & \omega_{f'}(\delta)\left(\frac{x(1-x)}{n\delta}+\sqrt{\sum\limits_{k=0}^n b_{n,k}\left(\frac{k}{n} -x\right)^2}\cdot\sqrt{\sum\limits_{k=0}^n b_{n,k}}\right)\\
& = & \omega_{f'}(\delta)\left(\frac{x(1-x)}{n\delta}+\sqrt{\frac{x(1-x)}{n}}\right)\\
& \leq & \omega_{f'}(\delta)\left(\frac{1}{4n\delta}+\frac{1}{2\sqrt{n}}\right)\\
& \leq &\frac{1}{\sqrt{n}} \omega_{f'}\left(\frac{1}{2\sqrt{n}}\right),
\end{eqnarray*}
where we have set $\delta=1/(2\sqrt{n})$. Thus, we can write
\begin{equation}\label{uw1}
|\mathcal{B}_n(\mathcal{K}f;x)-\mathcal{K}f|\leq \frac{1}{\sqrt{n}}~\omega_{(\mathcal{K}f)'}\left(\frac{1}{2\sqrt{n}}\right).
\end{equation}
Moreover, we have
\begin{eqnarray*}
\omega_{(\mathcal{K}f)'}\left(\frac{1}{2\sqrt{n}}\right)&=&\sup\left\{|(\mathcal{K}f)'(x)-(\mathcal{K}f)'(y)|:|x-y|\leq \frac{1}{2 \sqrt{n}};~x,y\in[0,1] \right\}\\
&=&\sup\left\{|f'(\phi(x))\phi'(x)-f'(\phi(y))\phi'(y)|:|x-y|\leq\frac{1}{2 \sqrt{n}};~x,y\in[0,1] \right\}
\end{eqnarray*}
and
\begin{equation*}
\begin{split}
|f'(\phi(x))\phi'(x)-f'(\phi(y))\phi'(y)| & =|f'(\phi(x))\phi'(x)-f'(\phi(y))\phi'(x)+f'(\phi(y))(\phi'(x)-\phi'(y))|\\
& \leq  |\phi'(x)| |f'(\phi(x))-f'(\phi(y))|+|f'(\phi(y))||\phi'(x)-\phi'(y)|.
\end{split}
\end{equation*}
Then, we get
\begin{equation*}
\begin{split}
\omega_{(\mathcal{K}f)'}\left(\frac{1}{2\sqrt{n}}\right) & \leq \underset{x\in [0,1]}\sup  
|\phi'(x)|\underset{x, y\in[0,1]}\sup\bigg\{\{ |f'(\phi(x))-f'(\phi(y))|: |x- y|\leq\frac{1}{ 2\sqrt{n}} \}\bigg\} \\
& \quad +\underset{x\in[0,1]}\sup\left\{|f'(\phi(x))|\right\} \, \underset{x, y\in[0,1]}\sup\left\{|(\phi'(x)-\phi'(y))|:|x-y|\leq\frac{1}{2\sqrt{n}}\right\} \\
& = L_{\phi} \, \omega_{\mathcal{K}f'}\left(\frac{1}{2\sqrt{n}}\right)  +\underset{x\in[0,1]}\sup\left\{|f'(\phi(x))|\right\} \, \frac{L_{\phi'}}{2\sqrt{n}} \\
\end{split}
\end{equation*}
and it follows from Lemma \ref{lem:modulus} that
\begin{equation*}
    \omega_{(\mathcal{K}f)'}\left(\frac{1}{2\sqrt{n}}\right) \leq {L_{\phi}} \, \omega_{f'}\left(\frac{L_{\phi}}{2\sqrt{n}}\right)\bigg|_{[\phi(0),\phi(1)]}+\underset{0\leq y\leq1}\sup |f'(\phi(y))| \, \frac{L_{\phi'}}{2\sqrt{n}}.
\end{equation*}
The result follows by combining the above inequality with \eqref{uw1}.
\end{proof}

This result provides a better convergence rate than Theorem \ref{Th1} (i.e. order $n^{-1}$ instead of $n^{-1/2}$), which is expected since the map and the observable are assumed to be continuously differentiable in the latter case. This property is illustrated with the following numerical example.

\begin{example}
\label{ex:uni}

The dynamics $\dot x=-x(1+x)$ generate the flow map $\phi_t(x)=\frac{x}{(e^t+x(e^t-1)}$ at any time $t$. For the map at $t=1$ and the function $f=x^2/2$, we compute the approximation error $\|\mathcal{B}_n\mathcal{K}f-\mathcal{K}f\|_\infty$ as a function of the number $n$ of Bernstein basis functions (Figure \ref{figc}). The error is compared with the upper bounds obtained in Theorem \ref{Th1} and Theorem \ref{Th2}. The results confirm the expected rates of convergence and show that a tighter bound is obtained when the first derivative of $f$ is considered.

\begin{figure}[h!]
    \centering 
    \includegraphics[width=.5\textwidth]{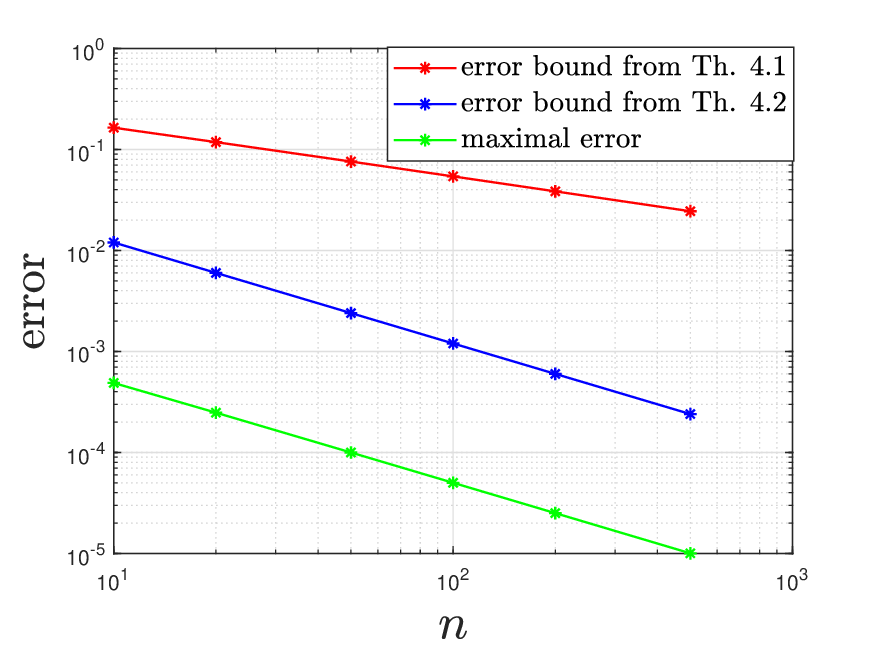}  
    \caption[]{Error bounds obtained with Theorem \ref{Th1} and Theorem \ref{Th2}, for the flow map considered in Example \ref{ex:uni}.}
    \label{figc}
\end{figure}
\end{example}

\subsection{Approximation of the Koopman operator in the multivariate case}

This subsection is devoted to the study of the Bernstein approximation of the Koopman operator in the multivariate case, in terms of the full and partial modulus of continuity.

\begin{theorem}\label{theo_multi}
Let $\phi:[0,1]^m \to [0,1]^m$ be a Lipschitz continuous map (with Lipschitz constant $L_\phi$). Then, for any $f\in C([0,1]^m)$, an upper bound on the error in approximating $\mathcal{K}f$ by $\mathcal{B}_\mathbf{n}\mathcal{K}f$ is given by
\begin{eqnarray*}
\|\mathcal{B}_\mathbf{n} \mathcal{K}f-\mathcal{K}f\|_\infty \leq\frac{3}{2} \left.\Omega_f\left(L_{\phi}\sqrt{\sum_{l=1}^m \frac{1}{n_l}}\right)\right|_{\phi([0,1]^m)}.
\end{eqnarray*}
\end{theorem}
\begin{proof} The proof is inspired from the proof given in \cite{Carothers} for the univariate case. For $f\in C[0,1]^m$ and $x\in [0,1]^m$, we obtain
\begin{eqnarray*}
\left|\mathcal{B}_\mathbf{n}(f;\textbf{x})-f(\textbf{x})\right|
&=&\left| \sum\limits_{{\substack{k_l=0 \\ l=1,\cdots,m}}
}^{n_l}f\left(\frac{k_1}{n_1},\frac{k_2}{n_2},\cdots, \frac{k_m}{n_m} \right) \left(\prod_{l=1}^{m}  b_{n_l,k_l}(x_l)\right)-f(\textbf{x})\right|\\
&=& \left|\sum\limits_{{\substack{k_l=0 \\ l=1,\cdots,m}}
}^{n_l}\left[f\left(\frac{k_1}{n_1},\frac{k_2}{n_2},\cdots, \frac{k_m}{n_m} \right)-f(x)\right] \prod_{l=1}^{m}  b_{n_l,k_l}(x_l)\right|\\
&\leq & \sum\limits_{{\substack{k_l=0 \\ l=1,\cdots,m}}
}^{n_l}\left|f\left(\frac{k_1}{n_1},\frac{k_2}{n_2},\cdots, \frac{k_m}{n_m} \right)-f(x)\right| \prod_{l=1}^{m}  b_{n_l,k_l}(x_l)~~\\
&\leq & \sum\limits_{{\substack{k_l=0 \\ l=1,\cdots,m}}
}^{n_l} \Omega_f\left(\sqrt{\sum\limits_{l=1}^m\left(\frac{k_l}{n_l}-x_l\right)^2} \right)\prod_{l=1}^{m}  b_{n_l,k_l}(x_l),
\end{eqnarray*}
where we used the partition of unity property \eqref{eq:Bernstein_constant}, the triangle inequality, and the definition of the modulus of continuity. Next, the property \eqref{eq:prop_modulus_nu} of the modulus of continuity yields
\begin{equation*}
\Omega_{f}\left(\sqrt{\sum\limits_{l=1}^m\left(\frac{k_l}{n_l}-x_l\right)^2} \right)
\leq  \Omega_{f}\left(\delta \right)\left(1+\frac{1}{\delta}\cdot \sqrt{\sum\limits_{l=1}^m\left(\frac{k_l}{n_l}-x_l\right)^2} \right)
\end{equation*}
for any $\delta>0$. It follows that, using \eqref{eq:Bernstein_constant} twice and the Cauchy-Schwartz inequality, we have
\begin{equation*}
\begin{split}
&|\mathcal{B}_\mathbf{n}(f;\textbf{x})-f(\textbf{x})|\\
& \, \leq \Omega_{f}\left(\delta \right) \left(\sum\limits_{{\substack{k_l=0 \\ l=1,\cdots,m}}
}^{n_l} \prod_{l=1}^{m}  b_{n_l,k_l}(x_l) +\frac{1}{\delta}\sum\limits_{{\substack{k_l=0 \\ l=1,\cdots,m}}
}^{n_l}  \sqrt{\sum\limits_{l=1}^m\left(\frac{k_l}{n_l}-x_l\right)^2}\prod_{l=1}^{m}  b_{n_l,k_l}(x_l) \right)\\ 
& \, \leq \Omega_{f}\left(\delta \right) \left(1+\frac{1}{\delta}\sum\limits_{{\substack{k_l=0 \\ l=1,\cdots,m}}}^{n_l} \left(\sum\limits_{l=1}^m\left(\frac{k_l}{n_l}-x_l\right)^2\prod_{l=1}^{m}  b_{n_l,k_l}(x_l)\right)^{1/2} \left( \prod_{l=1}^{m}  b_{n_l,k_l}(x_l)\right)^{\frac{1}{2}}   \right) \\
& \, \leq \Omega_{f}\left(\delta \right) \left(1+\frac{1}{\delta}\left(\sum\limits_{{\substack{k_l=0 \\ l=1,\cdots,m}}}^{n_l} \sum\limits_{l=1}^m\left(\frac{k_l}{n_l}-x_l\right)^2\prod_{l=1}^{m}  b_{n_l,k_l}(x_l)\right)^{\frac{1}{2}} \left(\sum\limits_{{\substack{k_l=0 \\ l=1,\cdots,m}}}^{n_l}  \prod_{l=1}^{m}  b_{n_l,k_l}(x_l)\right)^{\frac{1}{2}}  \right) \\
& \, \leq \Omega_{f}\left(\delta \right) \left(1+\frac{1}{\delta}\left(\sum\limits_{{\substack{k_l=0 \\ l=1,\cdots,m}}}^{n_l} \sum\limits_{l=1}^m\left(\frac{k_l}{n_l}-x_l\right)^2\prod_{l=1}^{m}  b_{n_l,k_l}(x_l)\right)^{\frac{1}{2}} \right).
\end{split}
\end{equation*}
Next, \eqref{eq:Bernstein_2nd_moment} leads to 
\begin{eqnarray*}
|\mathcal{B}_\mathbf{n}(f;\textbf{x})-f(\textbf{x})| 
&\leq & \Omega_{f}\left(\delta \right) \left(1+\frac{1}{\delta}\sqrt{\sum_{l=1}^m\frac{x_l(1-x_l)}{n_l}} \right)\\
&\leq & \Omega_{f}\left(\delta \right) \left(1+\frac{1}{2\delta}\sqrt{\sum_{l=1}^m\frac{1}{n_l}} \right)\\
&=& \frac{3}{2}\Omega_{f}\left(\sqrt{\sum_{l=1}^m \frac{1}{n_l}} \right)
\end{eqnarray*}
where we have set $\delta=\sqrt{\sum_{l=1}^m \frac{1}{n_l}}$.\\

In the context of the Koopman operator, since $\mathcal{K}f$ is continuous on $[0,1]^m$, we can write
\begin{equation*}\label{in3}
|\mathcal{B}_\mathbf{n}(\mathcal{K}f;\textbf{x})-\mathcal{K}f(\textbf{x})| \leq   \frac{3}{2}\Omega_{\mathcal{K}f}\left(\sqrt{\sum_{l=1}^m \frac{1}{n_l}} \right)
\end{equation*}
and it follows from Lemma \ref{lem:modulus} that
\begin{equation*}
    \Omega_{\mathcal{K}f}\left(\sqrt{\sum_{l=1}^m \frac{1}{n_l}} \right) \leq \left.\Omega_{f}\left(L_\phi \sqrt{\sum_{l=1}^m \frac{1}{n_l}} \right)\right|_{\phi([0,1]^m)},
\end{equation*}
which concludes the proof.
\end{proof}

By considering the partial modulus of continuity, we can also directly extend the results of the uni-dimensional case to the multivariate case. The next result extends Theorem \ref{Th1} to the multivariate case.
\begin{theorem}\label{theo_multi_2}
Let $\phi:[0,1]^m \to [0,1]^m$ be a Lipschitz continuous map (with partial Lipschitz constants $L^{(l)}_\phi$, $l=1,\dots,m$). Then, for any $f\in C([0,1]^m)$, an upper bound on the error in approximating $\mathcal{K}f$ by $\mathcal{B}_\mathbf{n}\mathcal{K}f$ is given by
\begin{eqnarray*}
\|\mathcal{B}_\mathbf{n}\mathcal{K}f-\mathcal{K}f\|_{\infty}\leq \frac{3}{2} \sum\limits_{l=1}^m \left.\Omega_f\left(\frac{L_\phi^{(l)}}{\sqrt{n_l}}\right)\right|_{\phi([0,1]^m)}.
\end{eqnarray*}
\end{theorem}
\begin{proof}
For a univariate function $f\in C([0,1])$, it follows from (the proof of) Theorem \ref{Th1} that 
\begin{equation*}
|\mathcal{B}_n(f;x)-f(x)| \leq \frac{3}{2} \omega_f\left(\frac{1}{\sqrt{n}}\right).
\end{equation*}
This trivially implies that, for $f\in C([0,1]^m)$,
\begin{equation*}
|\mathcal{B}^{(l)}_{n_l}(f;x_l)-f(x_l)| \leq \frac{3}{2} \Omega^{(l)}_f\left(\frac{1}{\sqrt{n_l}}\right)
\end{equation*}
with the operator $\mathcal{B}^{(l)}_{n_l}$ defined in Section \ref{sec:prelim_Bernstein}. Next, it follows from Lemma \ref{lem:partial} that
\begin{equation*}
|\mathcal{B}_\mathbf{n}(f;\textbf{x})-\mathcal{K}f(\textbf{x})| \leq \sum_{l=1}^m \overline{\mathcal{B}}^{(-l)}_{\mathbf{n}} \left(\frac{3}{2} \Omega^{(l)}_f\left(\frac{1}{\sqrt{n_l}}\right) \right) = \frac{3}{2} \sum_{l=1}^m \Omega^{(l)}_f\left(\frac{1}{\sqrt{n_l}}\right)
\end{equation*}
where we used \eqref{eq:Bernstein_constant} and the fact that $\Omega^{(l)}_f(1/\sqrt{n_l})$ is a constant. In the context of the Koopman operator, since $\mathcal{K}f$ is continuous, we have
\begin{equation}\label{k}
|\mathcal{B}_\mathbf{n}(\mathcal{K}f;\textbf{x})-\mathcal{K}f(\textbf{x})|\leq  \frac{3}{2} \sum\limits_{l=1}^m  \Omega^{(l)}_{\mathcal{K}f}\left(\sqrt{\frac{1}{n_l}}\right)
\end{equation}
and the result follows from Lemma \ref{lem:modulus}.
\end{proof}
It is noticeable that the least conservative bound can be obtained with either Theorem \ref{theo_multi} or \ref{theo_multi_2}, depending on the properties of the map $\phi$ and the number of Bernstein basis functions. This will illustrated in Example \ref{ex:multi}.

Finally, we extend the result of Theorem \ref{Th2}.
\begin{theorem}
\label{theo_multi_3}
Let $\phi:[0,1]^m \to [0,1]^m$ be a Lipschitz continuous map (with partial Lipschitz constant $L^{(l)}_{\phi}$) and whose partial derivatives $\partial_{l} \phi$ with respect to the $l$th variable are also Lipschitz continuous (with partial Lipschitz constant $L^{(l)}_{\partial_l \phi}$) for all $l$. Then, for any $f\in C^1([0,1]^m)$, an upper bound on the error in approximating $\mathcal{K}f$ by $\mathcal{B}_\mathbf{n}\mathcal{K}f$ is given by
\begin{equation*}
\|\mathcal{B}_\mathbf{n}\mathcal{K}f-\mathcal{K}f\|_\infty\leq \sum_{l=1}^m \frac{1}{\sqrt{n_l}} \left( L^{(l)}_{\phi} \,\Omega_{\nabla f}\left(\frac{L^{(l)}_{\phi}}{2\sqrt{n_l}}\right)\bigg|_{\phi([0,1]^m)}     +\underset{\textbf{x}\in\phi([0,1]^m)}\sup \|\nabla f(x)\| \, \frac{L^{(l)}_{\partial_l\phi}}{2\sqrt{n_l}}  \right).
\end{equation*}
\end{theorem}
\begin{proof}
    The inequality \eqref{uw1} in the proof of Theorem \ref{Th2} trivially yields
    \begin{equation*}
|\mathcal{B}^{(l)}_{n_l}(\mathcal{K}f;\textbf{x})-\mathcal{K}f(\textbf{x})| \leq \frac{1}{\sqrt{n_l}} \Omega^{(l)}_{\partial_l(\mathcal{K}f)} \left(\frac{1}{2\sqrt{n_l}}\right)
\end{equation*}
and following similar lines as in the proof of Theorem \ref{theo_multi_2}, we obtain
\begin{equation}
\label{eq:intermediate_ineq}
    |\mathcal{B}_\mathbf{n}(\mathcal{K}f;\textbf{x})-\mathcal{K}f(\textbf{x})|\leq \sum_{l=1}^m  \frac{1}{\sqrt{n_l}} \Omega^{(l)}_{\partial_l(\mathcal{K}f)} \left(\frac{1}{2\sqrt{n_l}}\right).
\end{equation}
Similarly to the proof of Theorem \ref{Th2}, we have
\begin{equation*}
\begin{split}
\Omega^{(l)}_{\partial_l(\mathcal{K}f)}\left(\frac{1}{2\sqrt{n_l}}\right)
& =\sup\left\{|\nabla f(\phi(\textbf{x})) \partial_l\phi(\textbf{x})-\nabla f(\phi(\textbf{y}))\partial_l\phi(\textbf{y})|:|x_l-y_l|\leq\frac{1}{2 \sqrt{n}};~\textbf{x},\textbf{y}\in[0,1]^m \right\}
\end{split}
\end{equation*}
and
\begin{equation*}
\begin{split}
|\nabla f(\phi(\textbf{x}))\partial_l \phi(\textbf{x})-\nabla f(\phi(\textbf{y}))\partial_l\phi(\textbf{y})|
&\leq  \|\partial_l\phi(\textbf{x})\| \|\nabla f(\phi(\textbf{x}))-\nabla f(\phi(\textbf{y}))\|+\|\nabla f(\phi(\textbf{y}))\|\|\partial_l\phi(\textbf{x})-\partial_l\phi(\textbf{y})\|.
\end{split}
\end{equation*}
Then, we obtain
\begin{equation*}
\begin{split}
\Omega^{(l)}_{\partial_l(\mathcal{K}f)}\left(\frac{1}{2\sqrt{n_l}}\right) & \leq \underset{\textbf{x}\in [0,1]^m}\sup  
\|\partial_l \phi(\textbf{x})\|\underset{\textbf{x}, \textbf{y}\in[0,1]^m}\sup\bigg\{\{ \|\nabla f(\phi(\textbf{x}))-\nabla f(\phi(\textbf{y}))\|: |x_l- y_l|\leq\frac{1}{ 2\sqrt{n_l}} \}\bigg\} \\
& \quad +\underset{\textbf{x}\in[0,1]^m}\sup\left\{\|\nabla f(\phi(\textbf{x}))\|\right\} \, \underset{\textbf{x}, \textbf{y}\in[0,1]^m}\sup\left\{\|(\partial_l \phi(\textbf{x})-\partial_l \phi(\textbf{y}))\|:|x_l-y_l|\leq\frac{1}{2\sqrt{n_l}}\right\} \\
& = L^{(l)}_{\phi} \, \Omega^{(l)}_{\mathcal{K}(\nabla f)}\left(\frac{1}{2\sqrt{n_l}}\right)  +\underset{\textbf{x}\in\phi([0,1]^m)}\sup\left\{\|\nabla f(\textbf{x})\|\right\} \, \frac{L^{(l)}_{\partial_l \phi}}{2\sqrt{n_l}} \\
& \leq L^{(l)}_{\phi} \, \Omega_{\nabla f}\left(\frac{L_\phi^{(l)}}{2\sqrt{n_l}}\right)  +\underset{\textbf{x}\in\phi([0,1]^m)}\sup\left\{\|\nabla f(\textbf{x})\|\right\} \, \frac{L^{(l)}_{\partial_l \phi}}{2\sqrt{n_l}}
\end{split}
\end{equation*}
with $\mathcal{K}(\nabla f)=\nabla f \circ \phi$ and where the last line follows from Lemma \ref{lem:modulus} (trivially adapted to a vector valued function). Finally, the result is obtained by combining the above inequality with \eqref{eq:intermediate_ineq}.
\end{proof}

\begin{example}
\label{ex:multi}
Consider the dynamics 
\begin{eqnarray*} 
\dot{x}_1&=&x_1(1+x_2)\\
\dot{x}_2&=& -x_2^2
\end{eqnarray*}
which generates a flow map
$$\phi_t(x_1,x_2)=\left(e^{t}x_1(tx_2+1),\frac{x_2}{1+tx_2}\right).$$
We compute the approximation error $\|\mathcal{B}_\mathbf{n} \mathcal{K}f-\mathcal{K}f\|_\infty$ for the flow map at $t=1$  and the function $f(x_1,x_2)=x_1^2 x_2^3$, for different numbers $n_1=n_2=n$ of Bernstein basis functions. This error is compared with the bounds estimated by using the full and partial modulus of continuity (Theorem \ref{theo_multi} and Theorem \ref{theo_multi_2}, respectively) along with Theorem \ref{theo_multi_3}. It is shown in Figure \ref{Error:multi_2} that, in this case, Theorem \ref{theo_multi_2} provides a better bound than Theorem \ref{theo_multi} only for large values $n$. Although it is more conservative for low values of $n$, the error bound obtained with Theorem \ref{theo_multi_3} is characterized by a better rate of convergence, thereby providing the best result for large values $n$. 
\begin{figure}[h!]
    \centering 
    \includegraphics[width=.5\textwidth]{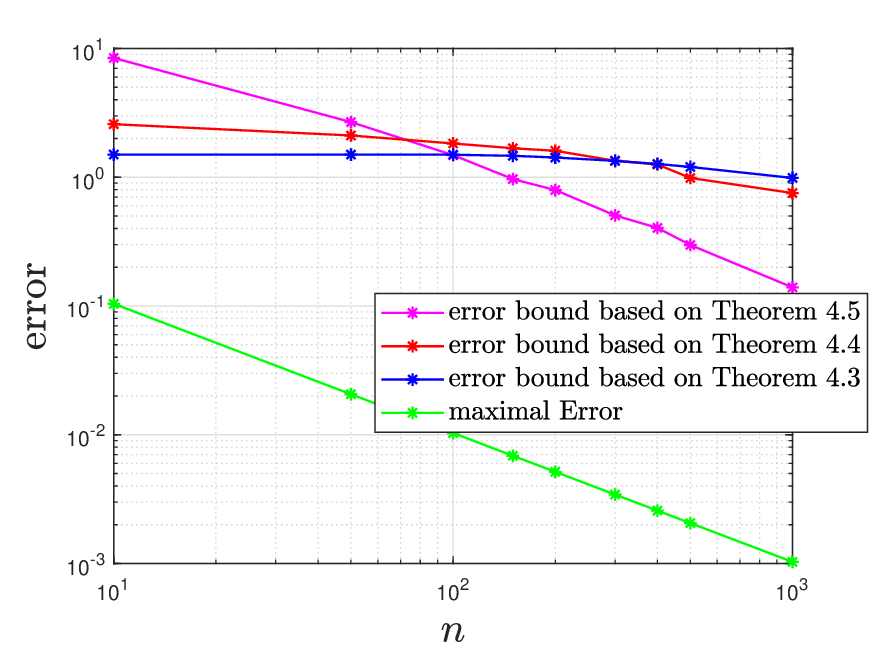}
    \caption[]{Error bounds computed with Theorem \ref{theo_multi} (full modulus of continuity), Theorem \ref{theo_multi_2} (partial modulus of continuity), and Theorem \ref{theo_multi_3} for the flow map considered in Example \ref{ex:multi}.}
    \label{Error:multi_2}
\end{figure}

\end{example}

\subsection{Error propagation under the iteration of the Koopman operator approximation}

In the previous subsections, we have computed error bounds for the Bernstein approximation of the Koopman operator. We will now consider the error propagation through the iteration of the Koopman operator approximation. We have the following general result.
    \begin{lemma} 
    \label{lem:iteration}
For $f\in C([0,1]^m)$ and $k\geq 1$, we have
 \begin{equation*}
        \|(\mathcal{B}_\mathbf{n}\mathcal{K})^kf-\mathcal{K}^kf\|_\infty \leq 
        \sum\limits_{j=0}^{k-1}\|(\mathcal{B}_\mathbf{n}\mathcal{K}(\mathcal{B}_\mathbf{n}\mathcal{K})^{j}f-\mathcal{K}(\mathcal{B}_\mathbf{n}\mathcal{K})^{j}f\|_\infty.
    \end{equation*}
\end{lemma}  
\begin{proof}
Using the triangle inequality, we obtain
    \begin{eqnarray*}
\|(\mathcal{B}_\mathbf{n}\mathcal{K})^{j}f-\mathcal{K}^{j}f\|_\infty&\leq & \|(\mathcal{B}_\mathbf{n}\mathcal{K}(\mathcal{B}_\mathbf{n}\mathcal{K})^{j-1}f-\mathcal{K}(\mathcal{B}_\mathbf{n}\mathcal{K})^{j-1}f\|_\infty\\
&&+\|\mathcal{K}(\mathcal{B}_\mathbf{n}\mathcal{K})^{j-1}f-\mathcal{K}^jf\|_\infty\\
&=& \|(\mathcal{B}_\mathbf{n}\mathcal{K}(\mathcal{B}_\mathbf{n}\mathcal{K})^{j-1}f-\mathcal{K}(\mathcal{B}_\mathbf{n}\mathcal{K})^{j-1}f\|_\infty\\
&&+\|\mathcal{K}\left((\mathcal{B}_\mathbf{n}\mathcal{K})^{j-1}f-\mathcal{K}^{j-1}\right)f\|_\infty\\
&\leq & \|(\mathcal{B}_\mathbf{n}\mathcal{K}(\mathcal{B}_\mathbf{n}\mathcal{K})^{j-1}f-\mathcal{K}(\mathcal{B}_\mathbf{n}\mathcal{K})^{j-1}f\|_\infty\\
&&+\|(\mathcal{B}_\mathbf{n}\mathcal{K})^{j-1}f-\mathcal{K}^{j-1}f\|_\infty
\end{eqnarray*}
where we have used the fact that $\|\mathcal{K}f\|_\infty \leq \|f\|_\infty$. Then, the result follows by induction.
\end{proof}

Combining this lemma with our previous results, we obtain the following theorem.
    \begin{theorem} 
    \label{th:iteration}
Let $\phi:[0,1]^m \to [0,1]^m$ be a Lipschitz continuous map (with Lipschitz constants $L_\phi$ and $L_\phi^{(l)}$). For $f\in C[0,1]^m$ and $k\geq 1$, we have
 \begin{eqnarray*}\|(\mathcal{B}_\mathbf{n}\mathcal{K})^kf-\mathcal{K}^kf\|_\infty & \leq & \frac{3}{2} \sum\limits_{j=0}^{k-1} \Omega_{(\mathcal{B}_\mathbf{n}\mathcal{K})^jf}\left(L_\phi\sqrt{\sum_{l=1}^m \frac{1}{n_l}}\right)  \\
\|(\mathcal{B}_\mathbf{n}\mathcal{K})^kf-\mathcal{K}^kf\|_\infty & \leq & \frac{3}{2} \sum\limits_{j=0}^{k-1}  \sum\limits_{l=1}^m\Omega_{(\mathcal{B}_\mathbf{n}\mathcal{K})^jf}\left(\frac{L_\phi^{(l)}}{\sqrt{n_l}}\right).
 \end{eqnarray*}
Moreover, if the partial derivatives $\partial_{l} \phi$ are also Lipschitz continuous (with partial Lipschitz constant $L^{(l)}_{\partial_l \phi}$) for all $l$, then, for $f\in C^1([0,1]^m)$ and $k\geq 1$, we have
\begin{equation*}
\|(\mathcal{B}_\mathbf{n}\mathcal{K})^{k}f-\mathcal{K}^{k}f\|_\infty\leq \sum\limits_{j=0}^{k-1} \sum_{l=1}^m \frac{1}{\sqrt{n_l}} \left( L^{(l)}_{\phi} \,\Omega_{\nabla ((\mathcal{B}_\mathbf{n}\mathcal{K})^jf)}\left(\frac{L^{(l)}_{\phi}}{2\sqrt{n_l}}\right)     + \|\nabla((\mathcal{B}_\mathbf{n}\mathcal{K})^jf)\|_{\infty} \, \frac{L^{(l)}_{\partial_l\phi}}{2\sqrt{n_l}}  \right).
\end{equation*}
 \end{theorem}
    \begin{proof}
    The result directly follows from Theorems \ref{theo_multi}, \ref{theo_multi_2}, and \ref{theo_multi_3}, combined with Lemma \ref{lem:iteration}.
    \end{proof}
    The inequality in Theorem \ref{th:iteration} depends on the modulus of continuity of the functions $(\mathcal{B}_\mathbf{n}\mathcal{K})^{j}f$ and their gradients, which are known polynomials. For the first inequality, we derive an alternative bound that only depends on the modulus of continuity of $f$, at the cost of higher conservativeness (see Lemma \ref{lem:alternative_bound} in Appendix \ref{ap:alternative}).

\section{A first step towards the data-driven setting}

According to \eqref{pr}, the Bernstein approximation $(\mathcal{B}_\mathbf{n}\mathcal{K})f$ requires the knowledge of the values of the function $f$ at $\phi(k_1/n_1,\dots,k_m/n_m)$. However, in a data driven context, the values of the map $\phi$ might not be given on a regular lattice, but at randomly distributed points over a compact set $D$. In this section, we pave the way to extend the Bernstein approximation method to the data-driven setting, by assuming the existence of an appropriate change of variables.

\subsection{Bernstein approximation of the Koopman operator over a set of random points}
\label{sec:Bernstein_data}

Suppose we are given a set $X \subset D$ of $N$ randomly distributed points
\begin{equation*}
X = \{x_j \in D: j = 1,\dots,N \},
\end{equation*}
with $N=\Pi_{j=1}^m (n_j+1)$ for some $n_1,\dots,n_m \in \mathbb{N}$.
We assume that the set $X$ can be seen as the set of vertices of a lattice that is isotopic to a regular lattice with vertices
\begin{equation*}
\tilde{X} = \left\{\left(\frac{k_1}{n_1},\dots,\frac{k_m}{n_m}\right) : (k_1,\dots,k_m) \in \{0,\dots, n_1\} \times \cdots \times \{0,\dots, n_m\} \right\}.
\end{equation*}
In particular, this implies that there exists a bijective map $S:\tilde{X} \to X$ such that there is no intersection between the edges 
$$\left[S\left(\frac{k_1}{n_1},\dots,\frac{k_l}{n_l},\dots,\frac{k_m}{n_m}\right),S \left(\frac{k_1}{n_1},\dots,\frac{k_l+1}{n_l},\dots,\frac{k_m}{n_m}\right)\right]$$ 
for all $(k_1,\dots,k_m)$ and $l$. It should be noted that the existence of $S$ is not obvious. For instance, it is trivially related to necessary conditions on the number of data points. The construction of the map might also be computationally expensive. These questions are left for future work.

Next, the map $S:\tilde{X} \to X$ can be extended to $S:[0,1]^m \to D$ through interpolation (see e.g. linear interpolation in Appendix \ref{ap:construct_S}). For instance, in the univariate case, the map $S$ satisfies $S(k/n)=x_{k+1}$, where the points $x_k$ are ordered so that $x_k < x_{k+1}$, and linear interpolation yields $S(x)=x_{k+1}+n(x_{k+2}-x_{k+1})(x-k/n)$ for $x\in[k/n,(k+1)/n]$.

We are now in position to define a modified Bernstein operator based on the approximation in the new variables $\tilde{\textbf{x}}=S^{-1}(\textbf{x}) \in [0,1]^m$.
\begin{definition}
    Suppose that $S:[0,1]^m \to D$ is a continuous bijective map. For $f\in C(D)$ and $\tilde{f} = f \circ S \in C([0,1]^m)$, we define the $m$-variate modified Bernstein polynomial by
    $$\tilde{\mathcal{B}}_\mathbf{n}(f;\textbf{x}) = \mathcal{B}_\mathbf{n}(\tilde{f};\tilde{\textbf{x}}) = \mathcal{B}_\mathbf{n}(f \circ S;S^{-1}(\textbf{x})).$$
    Moreover, we denote by $\tilde{\mathcal{B}}_\mathbf{n}:C(D)\to C(D)$ the associated modified Bernstein operator.
\end{definition}
The modified Bernstein operator associated with the map $S$ leads to a Koopman operator approximation that is well-suited to the data-driven context. More precisely, we can define
\begin{equation*}
(\tilde{\mathcal{K}} f)(\textbf{x})=\tilde{\mathcal{B}}_\mathbf{n}(\mathcal{K}f;\textbf{x}) = \mathcal{B}_\mathbf{n}(\mathcal{K}f \circ S;S^{-1}(\textbf{x})) =\sum_{j=1}^{n_1\cdots n_m} f \circ \phi  \left(S\left(\frac{\alpha_1(j)}{n_1},\cdots, \frac{\alpha_m(j)}{n_m} \right)\right) B_j(S^{-1}(\textbf{x})), 
\end{equation*}
where we used the same notation as in Section \ref{sec:matrix}, i.e. the map $\alpha : \mathbb{N} \rightarrow \mathbb{N}^m$ refers
to the lexicographic order and $B_j$ are multivariate Bernstein polynomials. We verify that the values of the map $\phi$ are used only at the points $S(\alpha_1(j)/n_1,\cdots, \alpha_m(j)/n_m)\in X$. Considering the pairs of data points $\{x_j,\phi(x_j)\}_{j=1}^N= \{x_j,y_j\}_{j=1}^N$ and the permutation map $\pi$ defined over $\{1,\dots,N\}$ such that $S(\alpha_1(j)/n_1,\cdots, \alpha_m(j)/n_m)= x_{\pi(j)}$, we finally obtain
\begin{equation}
\label{eq:data_Bernstein_approx}
(\tilde{\mathcal{K}} f)(\textbf{x}) =\sum_{j=1}^{n_1\cdots n_m} f(y_{\pi(j)}) \, B_j(S^{-1}(\textbf{x})).
\end{equation}

\begin{remark}
    The approximation based on the modified Bernstein operator $\tilde{\mathcal{B}}_\mathbf{n}$ can also be used when the map $\phi$ is known on a regular grid of points, but defined over a more general interval $[a_1, b_1] \times \cdots \times [a_m, b_m] \neq [0,1]^m$. In this case, the map $S:[0,1]^m \to [a_1, b_1] \times \cdots \times [a_m, b_m]$ corresponds to an affine transformation.
\end{remark}

\subsection{Approximation error}

We now extend our previous analysis of the approximation error to the data-driven setting considered in this section, where we use the modified Bernstein operator. This can be done through the following lemma.
\begin{lemma}
\label{lem:data}
    Let $f \in C(D)$ and let $S:[0,1]^m \to D$ be a continuous bijective map. Then,
    \begin{equation*}
        \|\tilde{\mathcal{B}}_\mathbf{n}f-f\|_{\infty} = \|\mathcal{B}_\mathbf{n}(f\circ S)-f\circ S\|_{\infty}.
    \end{equation*}
\end{lemma}
\begin{proof}
    We have
    \begin{equation*}
        \|\tilde{\mathcal{B}}_\mathbf{n}f-f\|_{\infty} = \|\mathcal{B}_\mathbf{n}(f\circ S;S^{-1}(\cdot))-f \circ S(S^{-1}(\cdot))\|_{\infty} = \|\mathcal{B}_\mathbf{n}(f\circ S)-f\circ S\|_{\infty}.
    \end{equation*}
\end{proof}
Combining with Theorem \ref{theo_multi} and Theorem \ref{theo_multi_2}, we have the following result.
\begin{corollary}
    Let $\phi:D \to D$ be a Lipschitz continuous map (with Lipschitz constant $L_\phi$) and $S:[0,1]^m \to D$ be a Lipschitz continuous bijective map (with Lipschitz constant $L_S$ and partial Lipschitz constant $L_S^{(l)}$). Then, for $f\in C(D)$, we have
\begin{equation*}
\|\tilde{\mathcal{B}}_\mathbf{n} \mathcal{K}f-\mathcal{K}f\|_\infty \leq\frac{3}{2} \Omega_f\left(L_{\phi} L_S \sqrt{\sum_{l=1}^m \frac{1}{n_l}}\right)
\end{equation*}
and
\begin{equation*}
\|\tilde{\mathcal{B}}_\mathbf{n} \mathcal{K}f-\mathcal{K}f\|_\infty \leq \frac{3}{2} \sum\limits_{l=1}^m \Omega_f\left(\frac{L_\phi L_S^{(l)}}{\sqrt{n_l}}\right).
\end{equation*}
\end{corollary}
\begin{proof}
    Using Lemma \ref{lem:data} with the continuous function $\mathcal{K}f$, we obtain
    \begin{equation*}
        \|\tilde{\mathcal{B}}_\mathbf{n} \mathcal{K}f-\mathcal{K}f\|_\infty = \|\mathcal{B}_\mathbf{n}(\mathcal{K}f\circ S)-\mathcal{K}f\circ S\|_{\infty} = \|\mathcal{B}_\mathbf{n}(f \circ \phi \circ S)- f \circ \phi \circ S\|_{\infty} = \|\mathcal{B}_\mathbf{n}\mathcal{K}_Sf- \mathcal{K}_Sf\|_{\infty},
    \end{equation*}
    with the Koopman operator $\mathcal{K}_Sf=f \circ \phi_S$ associated with the Lipschitz continuous map $\phi_S=\phi \circ S$ (with Lipschitz constant $L_\phi L_S$ and partial Lipschitz constant $L_\phi L_S^{(l)}$). Then, the result follows from Theorem \ref{theo_multi} and Theorem \ref{theo_multi_2} with $\mathcal{K}_S$ and $\phi_S$. 
\end{proof}
The above result provides a uniform bound on the Bernstein approximation of the Koopman operator that is valid on the whole state space $D$, while, in the present data-driven context, the flow is supposed to be known only at the data points.

Since $S$ maps the set $\tilde{X}$ onto $X$, it is clear that $L_S\geq L_S^{l} \geq n_l h$, where $h=\max_{x_i \in X} \min_{x_j \in X} \|x_i-x_j\|$ is the largest distance between neighboring data points. In particular, if the data points are distributed over $[0,1]^m$, i.e. $S:[0,1]^m \to [0,1]^m$, it follows that $h \geq 1/{n_l}$ so that $L_S \geq L_S^{(l)} \geq 1$ for all $l$. Moreover, the more unevenly distributed the data points are, the larger the Lipschitz constants are. This implies that the best approximation error bounds are obtained in the optimal case where the data points are uniformly distributed over a regular lattice. 

Finally, note that specific values of the Lipschitz constants can be computed when $S$ is defined through linear interpolation (see Appendix \ref{ap:construct_S}).

\subsection{Matrix approximation and comparison with EDMD}

Let us consider the subspace $\tilde{\mathsf{F}}_\mathbf{n}$ of polynomials of degree $\mathbf{n}$ in the variables $\tilde{\textbf{x}}=S^{-1}(\textbf{x})$. This subspace is spanned by the basis functions $B_j \circ S^{-1}$, where $B_j$ are the $m$-variate Bernstein basis polynomials. In this case, the matrix representation $\tilde{\mathbf{K}}_\mathcal{B}$ of the finite-dimensional Bernstein approximation $\tilde{\mathcal{K}}=\tilde{\mathcal{B}}_\mathbf{n}\mathcal{K}$ restricted to $\tilde{\mathsf{F}}_\mathbf{n}$ satisfies
\begin{equation*}
    \tilde{\mathcal{B}}_\mathbf{n}(\mathcal{K}(B\circ S^{-1});\textbf{x}) = \tilde{\mathcal{K}}(B\circ S^{-1})(\textbf{x}) = \tilde{\mathbf{K}}_\mathcal{B} \, B \circ S^{-1}(\textbf{x})
\end{equation*}
where we recall that $B$ is the vector of Bernstein basis polynomials. Following similar developments as in Section \ref{sec:matrix}, we have
\begin{equation*}
      [\tilde{\mathbf{K}}_\mathcal{B} B\circ S^{-1}(\textbf{x})]_i = \sum\limits_{j=1}^{n_1 \cdots n_m} B_i \circ S^{-1} \left(y_{\pi(j)}\right) B_j \circ S^{-1}(\textbf{x})
\end{equation*}
where we used \eqref{eq:data_Bernstein_approx}. It follows that the matrix approximation is given by
\begin{equation*}
    \tilde{\mathbf{K}}_\mathcal{B} = [B(S^{-1}(y_{\pi(1)})) \, B(S^{-1}(y_{\pi(2)})) \, \cdots \, B(S^{-1}(y_{\pi(N)}))],
\end{equation*}
or equivalently
\begin{equation*}
    \tilde{\mathbf{K}}_\mathcal{B} = \mathbf{C} \tilde{\mathbf{U}} \, 
\end{equation*}
with the data matrix
\begin{equation*}
    \tilde{\mathbf{U}} = [X(S^{-1}(y_{\pi(1)})) \, X(S^{-1}(y_{\pi(2)})) \, \cdots \, X(S^{-1}(y_{\pi(N)}))].
\end{equation*}
Note that the data points $x_j$ are implicitly used in the construction of the map $S$.
Moreover, we also have 
\begin{equation*}
    \tilde{\mathbf{K}}_\mathcal{B}^X = \tilde{\mathbf{U}} \mathbf{C}
\end{equation*}
in the basis $X(S^{-1}(\textbf{x}))=X(\tilde{\textbf{x}})$ of monomials in $\tilde{\textbf{x}}$.

The most popular data-driven method to approximate the Koopman operator is the so-called Extended Dynamic Mode Decomposition (EDMD) \cite{Williams}. For a given choice of basis functions (which we assume to be monomials here), the matrix approximation of the Koopman operator obtained with the EDMD method is given by
\begin{equation*}
\tilde{\mathbf{K}}_{\mathrm{EDMD}}=\mathbf{U_Y} \mathbf{U_X^+}
\end{equation*}
with the data matrices
\begin{equation*}
    \mathbf{U_X} = [X(x_1) \, X(x_2) \, \cdots \, X(x_N)], \qquad
    \mathbf{U_Y} = [X(y_1) \, X(y_2) \, \cdots \, X(y_N)] 
\end{equation*}
and where $\mathbf{U_X^+}$ denotes the Moore-Penrose pseudoinverse of $\mathbf{U_X}$. It is noticeable that the computation of the matrix pseudoinverse is not required for the Bernstein approximation. On the other hand, some complexity is added through the possibly intricate construction of the map $S$ and the computation of its inverse (through linear interpolation, which is typically less expensive than computing the pseudoinverse). Note also that, when $S$ is the identity (uniformly distributed data points over $[0,1]^m$), we have $\mathbf{U_Y}=\tilde{\mathbf{U}}=\mathbf{U}$. In this case, $\mathbf{\tilde{K}}_{\mathrm{EDMD}}=\mathbf{U_Y} \mathbf{U_X^+}$ and $\mathbf{K}_\mathcal{B}^X = \mathbf{U_Y} \mathbf{C}$ are both matrix approximations of the Koopman operator in the same basis of monomials.

In the following example, we compare the performance of the Bernstein approximation method and the EDMD method in a data-driven context.

\begin{example}
Consider the competitive Lotka-Volterra dynamics
\begin{eqnarray*}
\dot{x}_1&=&1.5 \, x_1(1-x_1)-x_1 \, x_2\\
\dot{x}_2&=&1.5 \, x_2(1-x_2)- x_1 \, x_2
\end{eqnarray*}
and its flow map $\phi$ generated at time $t=1$. The true trajectory starting from the initial condition $x_0=(0.4,0.3)$ is compared with estimated trajectories obtained by iterating the Bernstein matrix approximation $\tilde{\mathbf{K}}^X_\mathcal{B}$ and the EDMD matrix approximation $\tilde{\mathbf{K}}_{\mathrm{EDMD}}$. Both approximation matrices were computed with $N=256$ data points shown in Figure \ref{fig:Lotka_data}(a) and with the same subspace $\mathsf{F}_\mathbf{n}$ of polynomials of degree less or equal to $15$ (i.e. $(n_1,n_2)=(15,15)$). For the Bernstein approximation, the map $S$ was constructed with a linear interpolation based on Delaunay triangulation of the data points. While the Bernstein approximation yields an accurate prediction of the trajectory, the EDMD method produces a trajectory that rapidly diverges. Note that the error bounds are very conservative in this case due to the low values $n_1$ and $n_2$. Better bounds could be obtained with polynomials of higher degree (and more data points), but this would yield unstable Bernstein matrices and diverging trajectories in long-term prediction. Additionally, the values of the flow map were corrupted with Gaussian noise (with zero mean and component-wise standard deviation equal to $0.02$). The results show that the Bernstein approximation is little affected by noise measurement, in contrast to EDMD.

While this numerical example suggests that the Bernstein approximation yields better prediction results than the EDMD method used with a polynomial basis, it does not allow to draw a general conclusion, nor does it imply that EDMD might not have better performance with other sets of basis functions.

\begin{figure}[h]
     \centering  
		\subfigure[Irregular lattice of data points]{\includegraphics[width=0.45\textwidth]{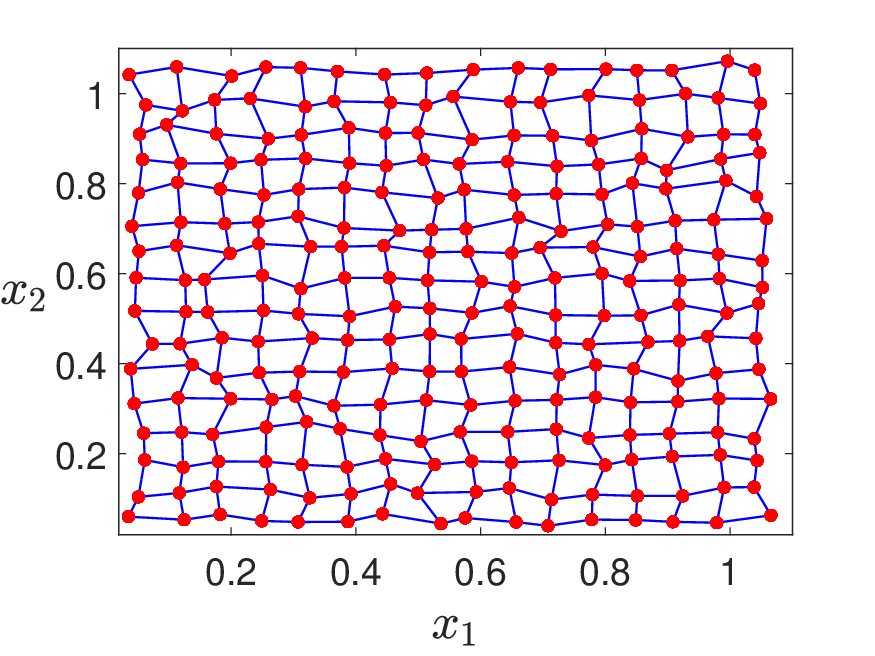}}
		\subfigure[Predicted trajectories]{\includegraphics[width=0.45\textwidth]{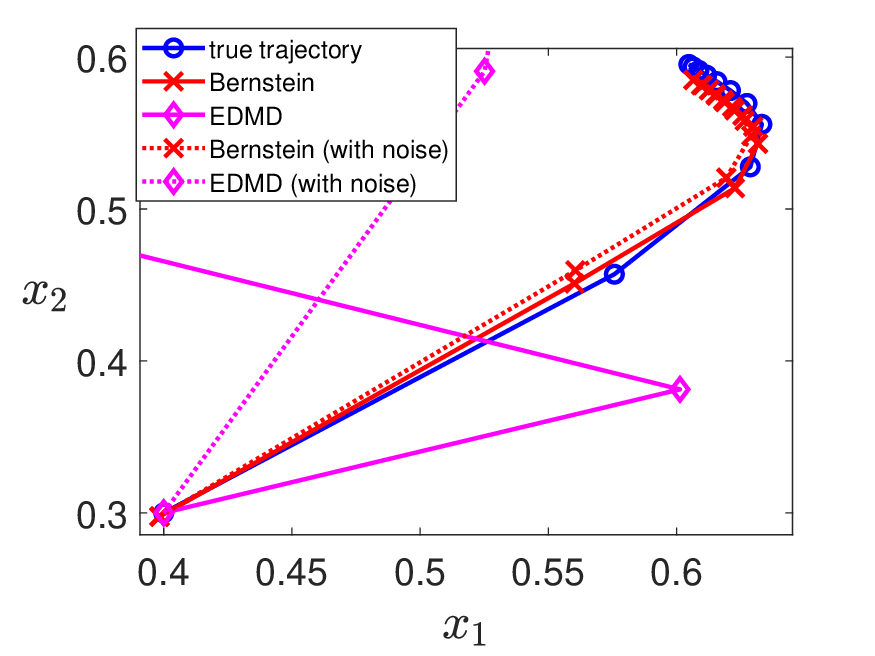}}
    \caption{For the competitive Lotka-Volterra system, Bernstein and EDMD approximations of the Koopman operator are computed from the values of the flow at the data points (panel (a)). Predicted trajectories are obtained by iterating the approximation matrices (panel (b)). The Bernstein approximation yields an accurate prediction, even in the case of measurement noise (dashed curves).}
    \label{fig:Lotka_data}
\end{figure}
\end{example}

\section{Conclusions and Perspectives}

In this paper, we have developed a novel finite-dimensional approximation scheme for the Koopman operator, which is based on Bernstein polynomial approximation. For several cases (univariate and multivariate maps), the method was complemented with an error analysis, providing convergence rates and approximation error bounds in the uniform sense which are inherited from the properties of Bernstein approximation. The errors bounds are expressed in terms of the modulus of continuity of the observables (and possibly of their derivative) and requires the sole knowledge of the Lipschitz constant of the map defining the Koopman operator. In addition, through an appropriate change of variables, the framework has been extended to a data-driven setting, where the flow is known at randomly distributed points, and it has been compared to the EDMD method in the context of prediction.

The present work opens several perspectives. First, the approximation error bounds obtained in this paper appear to be quite conservative. 
In this context, tighter bounds could be obtained by using the modulus of continuity of (higher-order) derivatives of the observables. Similarly, better convergence rates could be guaranteed through the use of iterated Bernstein polynomials (see e.g. \cite{Guan, Kelisky}).
In the same line, the use of other operators could be investigated, such as the Sz\'{a}sz–Mirakyan operator extending the Bernstein approximation to unbounded sets, or the Bernstein-Kantorovich operator extending the approximation to discontinuous integrable functions. Moreover, in the data-driven context, the method relies on a map which ``preserves the lattice structure'' from a regular lattice over $[0,1]^m $ to the set of data points. Both existence and algorithmic construction of such a map are left as an open problem. Finally, the efficiency and relevance of the proposed Bernstein approximation of the Koopman operator should be further investigated, for instance in the context of high-dimensional systems, spectral analysis of dynamical systems, and nonlinear control theory.

\appendix

\section{Alternative bound in Theorem \ref{th:iteration}}
\label{ap:alternative}
We have the following result, which provides an alternative bound in the first inequality of Theorem \ref{th:iteration}.
\begin{lemma}
\label{lem:alternative_bound}
    For $f \in C([0,1]^m)$, $k\geq 1$, and $\delta=\sqrt{\sum_{l=1}^m \frac{1}{n_l}}$, we have
    \begin{equation*}
         \sum\limits_{j=0}^{k-1} \Omega_{(\mathcal{B}_\mathbf{n}\mathcal{K})^jf}(L_\phi\delta) \leq \sum\limits_{l=1}^{k} 4^{k-l} \, \Omega_{f}(L_\phi^{l}\delta) .
    \end{equation*}
\end{lemma}

\begin{proof}
We first show that
       \begin{equation}\label{cor:ineq}
           \Omega_{(\mathcal{B}_\mathbf{n}\mathcal{K})^{j}f}(L_\phi \delta) \leq \Omega_{f}(L_\phi^{j+1}\delta) + 3 \sum_{l=1}^j 4^{j-l} \Omega_{f}(L_\phi^{l}\delta).
       \end{equation}
       Using the definition of the full modulus of continuity, we have
        \begin{eqnarray*}
          \Omega_{(\mathcal{B}_\mathbf{n}\mathcal{K})^{j}f}(L_\phi\delta) = \max\{|(\mathcal{B}_\mathbf{n}\mathcal{K})^{j}f(\textbf{x})-(\mathcal{B}_\mathbf{n}\mathcal{K})^{j}f(\textbf{y})|:\|\textbf{x}-\textbf{y}\|\leq L_\phi\delta\}
        \end{eqnarray*}
        and the triangle inequality yields
        \begin{equation*}
            \begin{split}
                |(\mathcal{B}_\mathbf{n}\mathcal{K})^{j}f(\textbf{x})-(\mathcal{B}_\mathbf{n}\mathcal{K})^{j}f(\textbf{y})|
                & \leq |\mathcal{B}_\mathbf{n}\mathcal{K}(\mathcal{B}_\mathbf{n}\mathcal{K})^{j-1}f(\textbf{x})-\mathcal{K}(\mathcal{B}_\mathbf{n}\mathcal{K})^{j-1}f(\textbf{x})|\\ 
                &\quad+|\mathcal{B}_\mathbf{n}\mathcal{K}(\mathcal{B}_\mathbf{n}\mathcal{K})^{j-1}f(\textbf{y})-\mathcal{K}(\mathcal{B}_\mathbf{n}\mathcal{K})^{j-1}f(\textbf{y})|\\
                &\quad+ |\mathcal{K}(\mathcal{B}_\mathbf{n}\mathcal{K})^{j-1}f(\textbf{x})-\mathcal{K}(\mathcal{B}_\mathbf{n}\mathcal{K})^{j-1}f(\textbf{y})|.
            \end{split}
        \end{equation*}
         Maximizing and using Theorem \ref{theo_multi} and Lemma \ref{lem:modulus}, we obtain
        \begin{eqnarray*}
            \Omega_{(\mathcal{B}_\mathbf{n}\mathcal{K})^{j}f}(L_\phi\delta) &\leq& \frac{3}{2}\Omega_{(\mathcal{B}_\mathbf{n}\mathcal{K})^{j-1}f}(L_{\phi}\delta)+\frac{3}{2}\Omega_{(\mathcal{B}_\mathbf{n}\mathcal{K})^{j-1}f}(L_{\phi}\delta)+\Omega_{\mathcal{K}(\mathcal{B}_\mathbf{n}\mathcal{K})^{j-1}f}(L_{\phi}\delta)\\
            &\leq& 3\Omega_{(\mathcal{B}_\mathbf{n}\mathcal{K})^{j-1}f}(L_{\phi}\delta) + \Omega_{(\mathcal{B}_\mathbf{n}\mathcal{K})^{j-1}f}(L_{\phi}^2\delta)
        \end{eqnarray*}
        and similarly
           \begin{equation*}
            \Omega_{(\mathcal{B}_\mathbf{n}\mathcal{K})^{j}f}(L_\phi^l\delta) \leq 3\Omega_{(\mathcal{B}_\mathbf{n}\mathcal{K})^{j-1}f}(L_{\phi}\delta) + \Omega_{(\mathcal{B}_\mathbf{n}\mathcal{K})^{j-1}f}(L_{\phi}^{l+1}\delta). 
        \end{equation*}
        It follows by recursion that
        \begin{equation}
        \label{eq:recursion}
            \Omega_{(\mathcal{B}_\mathbf{n}\mathcal{K})^{j}f}(L_\phi\delta) \leq \Omega_{f}(L_\phi^{j+1}\delta) + 3 \sum_{l=0}^{j-1} \Omega_{(\mathcal{B}_\mathbf{n}\mathcal{K})^{l}f}(L_\phi\delta).
        \end{equation}
        By comparing with \eqref{cor:ineq}, we see that it remains to show that 
        \begin{equation}
        \label{eq:inequal_recurs}
            \sum_{l=0}^{j-1} \Omega_{(\mathcal{B}_\mathbf{n}\mathcal{K})^{l}f}(L_\phi\delta) \leq \sum_{l=1}^j 4^{j-l} \Omega_{f}(L_\phi^{l}\delta).
        \end{equation}
        We proceed with a recursive argument. It is trivial that \eqref{eq:inequal_recurs} holds for $j=1$.  Now we suppose that it is true for $j=\bar{j}$. Using \eqref{eq:recursion} and \eqref{eq:inequal_recurs}, we have
        \begin{equation*}
        \begin{split}
            \sum_{l=0}^{\bar{j}} \Omega_{(\mathcal{B}_\mathbf{n}\mathcal{K})^{l}f}(L_\phi\delta) & \leq \sum_{l=0}^{\bar{j}-1} \Omega_{(\mathcal{B}_\mathbf{n}\mathcal{K})^{l}f}(L_\phi\delta) + \Omega_{(\mathcal{B}_\mathbf{n}\mathcal{K})^{\bar{j}}f}(L_\phi\delta) \\
            & \leq 4 \sum_{l=0}^{\bar{j}-1} \Omega_{(\mathcal{B}_\mathbf{n}\mathcal{K})^{l}f}(L_\phi\delta) + \Omega_{f}(L_\phi^{\bar{j}+1}\delta) \\
            & \leq 4 \sum_{l=1}^{\bar{j}} 4^{\bar{j}-l} \Omega_{f}(L_\phi^{l}\delta)
            + \Omega_{f}(L_\phi^{\bar{j}+1}\delta) \\
            & \leq \sum_{l=1}^{\bar{j}+1} 4^{\bar{j}+1-l} \Omega_{f}(L_\phi^{l}\delta)
            \end{split}
        \end{equation*}
    so that \eqref{eq:recursion} is valid for $\bar{j}+1$.
    Finally, using \eqref{cor:ineq}, we compute
    \begin{equation*}
    \begin{split}
        \sum\limits_{j=0}^{k-1} \Omega_{(\mathcal{B}_\mathbf{n}\mathcal{K})^jf}(L_\phi\delta) & \leq \sum\limits_{j=0}^{k-1} \Omega_{f}(L_\phi^{j+1}\delta) + 3 \sum\limits_{j=0}^{k-1}  \sum_{l=1}^j 4^{j-l} \Omega_{f}(L_\phi^{l}\delta) \\
        & = \sum\limits_{j=0}^{k-1} \Omega_{f}(L_\phi^{j+1}\delta) + 3 \sum\limits_{l=1}^{k-1} \Omega_{f}(L_\phi^{l}\delta) \sum_{j=l}^{k-1} 4^{j-l} \\
        & = \sum\limits_{l=1}^{k} \Omega_{f}(L_\phi^{l}\delta) +\sum\limits_{l=1}^{k-1} (4^{k-l}-1) \, \Omega_{f}(L_\phi^{l}\delta)  \\
        & = \sum\limits_{l=1}^{k} 4^{k-l} \, \Omega_{f}(L_\phi^{l}\delta) . 
        \end{split}
    \end{equation*}  
\end{proof}
Alternative bounds could also be derived for the other inequalities in Theorem \ref{th:iteration} by following similar lines. Since these bounds are conservative and might have a cumbersome expression, they are not given here.

\section{Construction of the map $S$}
\label{ap:construct_S}

We discuss the extension of the map $S:\tilde{X} \to X$ to $S:[0,1]^m \to D$ through linear interpolation, where we assume here that $D$ is the convex hull of $X$. 
Suppose that $\textbf{x}\in [0,1]^m$ lies in the $m$-simplex $C^{(j)}$ with vertices $v^{(j)}_0,\dots,v^{(j)}_{m} \in \tilde{X}$ and is characterized by the barycentric coordinates $(\beta_0,\dots,\beta_{m})\in [0,1]^{m+1}$ so that $x = \beta_0 v^{(j)}_0 + \cdots + \beta_m v^{(j)}_m$. In particular, we can compute
\begin{equation}
\label{eq:bary}
\begin{split}
    & \begin{pmatrix}
        \beta_1 \\ \vdots \\ \beta_m
    \end{pmatrix} = \begin{pmatrix}
    | & & | \\
        v^{(j)}_{1}-v^{(j)}_{0} & \cdots & v^{(j)}_{m}-v^{(j)}_{0} \\
        | & & |
    \end{pmatrix}^{-1}
        \left(\textbf{x}-v^{(j)}_{0}\right) \\
        & \quad \beta_0=1-\beta_1-\cdots-\beta_m.
        \end{split}
\end{equation}
Next, it follows from linear interpolation that 
\begin{equation*}
    S(\textbf{x}) = \beta_0 \, S(v^{(j)}_0) + \cdots + \beta_m \, S(v^{(j)}_m),
\end{equation*}
which can be rewritten, using \eqref{eq:bary}, as
\begin{equation}
\label{eq:map_S}
\begin{split}
S(\textbf{x}) & = S(v^{(j)}_0) + \begin{pmatrix}
| & & | \\
S(v^{(j)}_1)-S(v^{(j)}_0) & \cdots & S(v^{(j)}_m)-S(v^{(j)}_0) \\
| & & |
\end{pmatrix}
\begin{pmatrix}
    | & & | \\
        v^{(j)}_{1}-v^{(j)}_{0} & \cdots & v^{(j)}_{m}-v^{(j)}_{0} \\
        | & & |
    \end{pmatrix}^{-1}
        \left(\textbf{x}-v^{(j)}_{0}\right) \\
        & \triangleq S(v^{(j)}_0) + \mathbf{S}^{(j)}(\mathbf{V}^{(j)})^{-1} \left(\textbf{x}-v^{(j)}_{0}\right).
        \end{split}
\end{equation}
Note that $S(v_l^{(j)})\in X$ is well-defined for all $l \in \{0,\dots,m\}$  since $v_l^{(j)} \in \tilde{X}$. Moreover, if the vertices $v_l^{(j)}$ are adjacent points of the regular lattice (i.e. $(v_l^{(i)}-v_l^{(0)})^T(v_l^{(j)}-v_l^{(0)})=1/n_i \, \delta_{ij}$), we have
\begin{equation*}
 (\mathbf{V}^{(j)})^{-1} = \mathrm{diag}(n_1, \dots, n_m).
\end{equation*}
Proceeding along the same lines for all simplices $C^{(j)}$ such that $\cup_j C^{(j)} = [0,1]^m$, we finally obtain a piecewise-linear map $S$ over $[0,1]^m$.

Next, the Lipschitz constant of $S$ is computed as
\begin{equation*}
L_S = \max_{C^{(j)}} \left\| \mathbf{S}^{(j)}(\mathbf{V}^{(j)})^{-1} \right\|_2
\end{equation*}
where the maximum is taken over all simplices $C^{(j)}$ used in the interpolation process. Similarly, the partial Lipschitz constant is given by
\begin{equation*}
L_S^{(l)} = \max_{C^{(j)}} \left\| \mathbf{S}^{(j)}[(\mathbf{V}^{(j)})^{-1}]_{:,l} \right\|
\end{equation*}
where $[\cdot]_{:,l}$ denotes the $l$th column of a matrix.

\normalem
\printbibliography

\end{document}